\documentclass[a4paper,11pt]{article}

\usepackage[all,cmtip]{xy}
\usepackage{a4wide}
\usepackage{amsmath, amsthm, amssymb}
\usepackage{fontenc}
\usepackage[latin1]{inputenc}
\usepackage[english]{babel}
\usepackage{graphicx}
\usepackage[hang]{subfigure}

\usepackage{enumerate}

\newcommand{\cone}{\mathcal{C}}

\newtheorem{theorem}{Theorem}[section]

\newtheorem{proposition}[theorem]{Proposition}
\newtheorem{lemma}[theorem]{Lemma}
\newtheorem{remark}[theorem]{Remark}
\newtheorem{algorithm}[theorem]{Algorithm}
\newtheorem{example}[theorem]{Example}
\newtheorem{corollary}[theorem]{Corollary}

\newcommand{\NN}{\mathbb{N}}

\newcommand{\RR}{\mathbb{R}}

\newcommand{\R}{\mathbb{R}}

\def\tT{{\mbox{\tiny{T}}}}
\def\argmin{\mathop{\rm argmin}}

\newcommand{\supp}{\text{supp}\,}

\newcommand{\fftn}{\text{\texttt{fft2}}}

\newcommand{\ifftn}{\text{\texttt{ifft2}}}

\title{Convex Multiclass Segmentation with Shearlet Regularization}
\date{\today}
\author{
S. H\"auser
and
G. Steidl\thanks{University Kaiserslautern, Dept. of Mathematics, Kaiserslautern, Germany}
}

\begin{document}

\maketitle

\begin{abstract}
\noindent
Segmentation plays an important role in many preprocessing stages in image processing.
Recently, convex relaxation methods for image multi-labeling were proposed
in the literature. Often these models involve the total variation (TV) semi-norm
as regularizing term. However, it is well-known that the TV functional
is not optimal for the segmentation of textured regions.
In recent years directional representation systems were proposed to cope with
curved singularities in images. In particular, curvelets and shearlets provide
an optimally sparse approximation in the class of piecewise smooth functions with
$C^2$ singularity boundaries.
In this paper, we demonstrate that
the discrete shearlet transform is suited as regularizer for the segmentation of
curved structures.
Neither the shearlet nor the curvelet transform where used as regularizer
in a segmentation model so far.
To this end, we have implemented a translation invariant
finite discrete shearlet transform based on the FFT.
We describe how the shearlet transform can be incorporated within
the multi-label segmentation model and show how to find
a minimizer of the corresponding functional by applying an alternating direction method
of multipliers. Here the Parseval frame property of our shearlets comes into play.
We demonstrate by numerical examples that the shearlet regularized model
can better segment curved textures than the TV regularized one and
that the method can also cope with regularizers obtained from non-local means.
\end{abstract}
\section{Introduction} \label{sec_intro}
In recent years, much effort has been spent to design
directional representation systems for images
such as curvelets \cite{CDDY06}, ridgelets \cite{CD99}
and  shearlets  \cite{KGL06} and corresponding transforms (this list is not complete).
Among these transforms, the shearlet transform stands out since it stems
from a square-integrable group representation
\cite{DKMSST08} and has  the corresponding useful mathematical properties.
Moreover, similarly as wavelets are related to Besov spaces via atomic decompositions,
shearlets correspond to certain function spaces, the so-called shearlet coorbit spaces
\cite{DKST09}. In addition shearlets provide
an optimally sparse approximation in the class of piecewise smooth functions with $C^2$ singularity curves,
i.e.,
$$
\|f - f_N\|_{L_2}^2 \le C  N^{-2} (\log N)^3 \quad {\rm as} \; N \rightarrow \infty,
$$
where $f_N$ is the nonlinear shearlet approximation of a function $f$
from this class obtained by taking the $N$ largest shearlet coefficients in absolute value.

In the following, we want to show how the directional information encoded by the shearlet transform
can be used in image segmentation.
Although shearlets and curvelets have been applied to a wide field of image processing tasks, see, e.g., \cite{GL09,KL10,ELL08,MP09_a,MP09,YLEK09},
to the best of our knowledge they were not used in connection with segmentation
so far.
Segmentation is a fundamental task in image processing which plays
a role in many preprocessing steps.
Recently, convex relaxation methods for image multi-labeling were addressed by several authors
\cite{BYT09,EG11,JKS07,KSY11,LBS09,LKYBS09,LS10,PCCB09,SKC10,ZGFN08}. In
 these methods a convex functional consisting of a data term
and a regularization term is minimized.
Due to its edge-preserving properties a frequent choice for the regularization term is the
total variation (TV) semi-norm introduced by Rudin, Osher and Fatemi \cite{ROF92}.
The data term is either related to a predefined codebook (of gray/RGB-values related to the labels)
or the codebook is updated during the minimization process itself. For the later approach see, e.g., \cite{BYT11,CJPT10,HHMSS11}.
In this paper, we focus on the first approach, but replace the TV regularizer by the shearlet transform.
To this end, we introduce a simple discrete shearlet transform which translates the shearlets over the full grid
at each scale and for each direction. Using the FFT this transform can be still realized in a fast way.
We show by numerical examples that the shearlet regularized model is
well-suited to segment curved textures and can even cope with models incorporating nonlocal
means.
\begin{figure}[htbp]
    \centering
    \subfigure[Forms with different edge orientations]{\label{fig:formen} \includegraphics[width=0.3\textwidth]{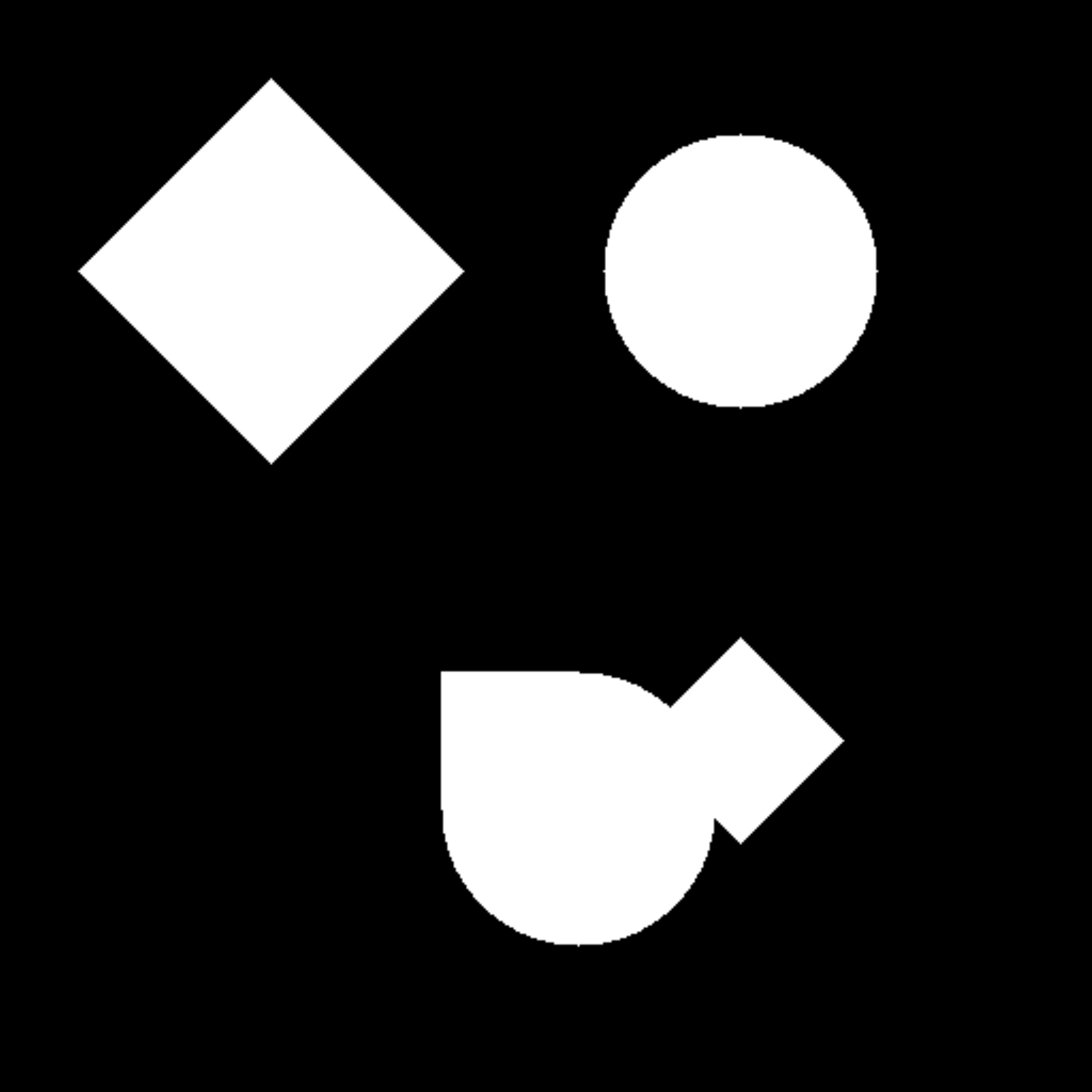}}
\hspace{0.2cm}
    \subfigure[Shearlet coefficients \newline for $a=\frac{1}{64}$ and $s = -1$]{\label{fig:formenShearlet} \includegraphics[width=0.3\textwidth]{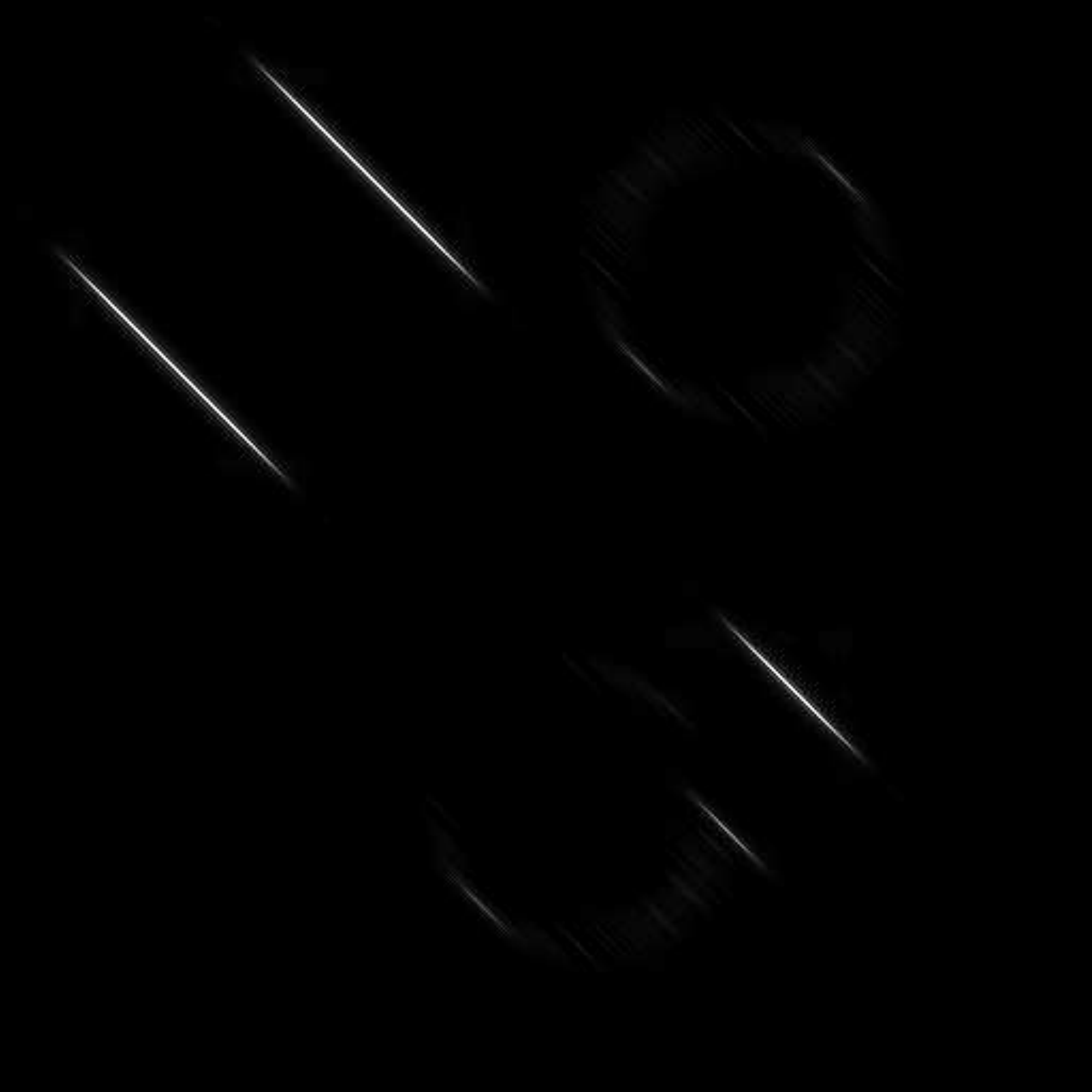}}
\hspace{0.2cm}
    \subfigure[Sum of shearlet coefficients for $a=\frac{1}{64}$ for all $s$]{\label{fig:formenShearletSumme}\includegraphics[width=0.3\textwidth]{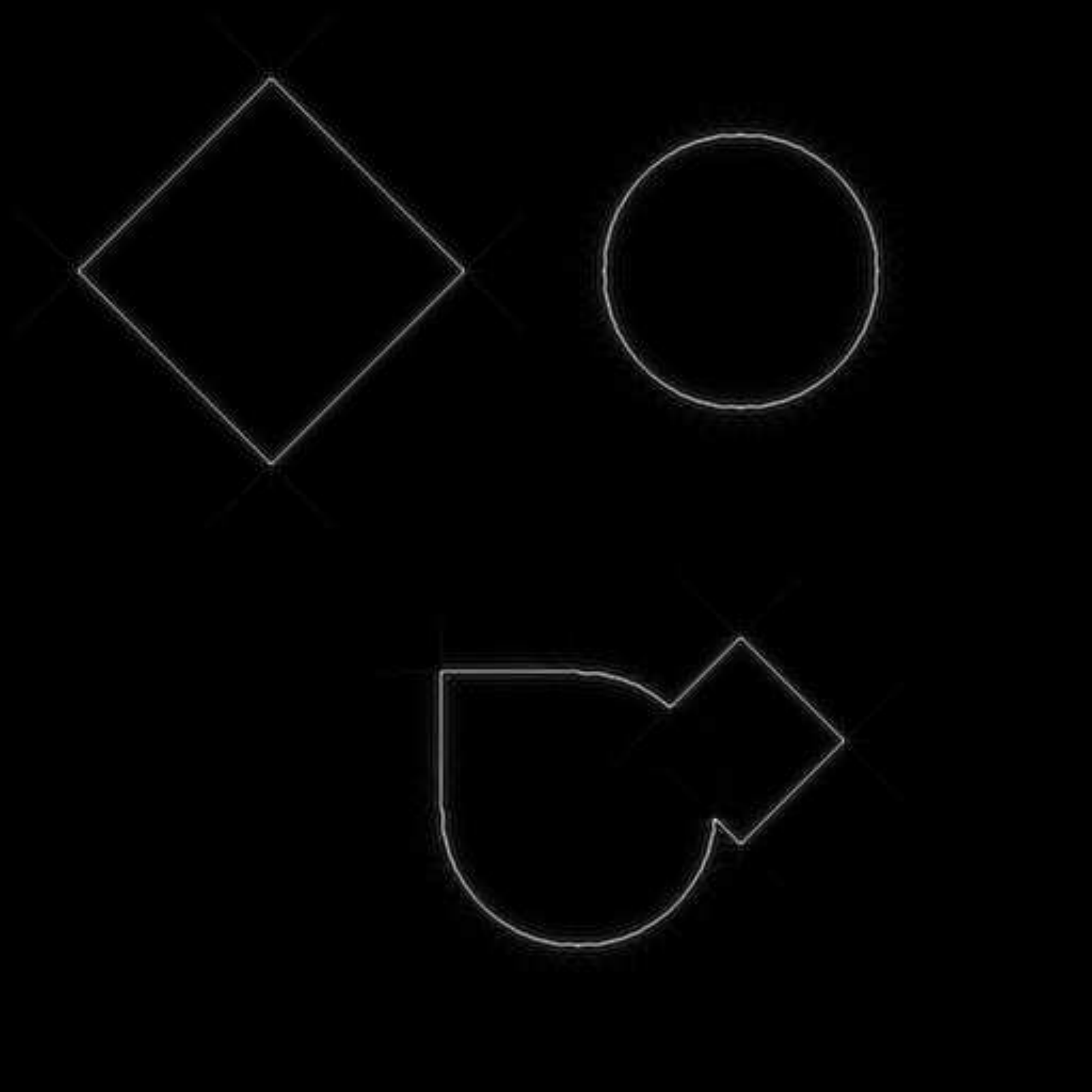}}
    \caption{Shearlet coefficients can detect edges with different orientations.}
    \label{fig:edgeDetection}
\end{figure}

This paper is organized as follows:
In Section \ref{sec_dst} we introduce the translation invariant, finite, discrete shearlet transform
whose full description cannot be found in the literature so far.
Here we follow the path via the continuous shearlet transform,
its counterpart on cones and finally its discretization on the full grid.
This is different to other implementations as, e.g., in \emph{ShearLab}\footnote{http://www.shearlab.org},
see \cite{SKZ11}.
Our discrete shearlet transform can be efficiently computed by the fast Fourier transform (FFT).
The discrete shearlets constitute a Parseval frame of the finite Euclidean space such that
the inversion of the shearlet transform  can be simply done by applying the adjoint transform.
This is proven in Theorem \ref{theorem:discreteShearletFrame} and cannot be found
in the literature so far.
In Section \ref{sec_segm_shearlets} we recall a convex multi-label segmentation model
and show how the shearlet transform can be incorporated into this approach.
We propose the method of alternating direction of multipliers (ADMM)
\cite{BPCPE11,EB92,EZC10,Gab83}
to find a minimizer of the functional
where the Parseval frame property of our shearlet frame can be exploited.
Numerical results are presented in Section \ref{sec:num_results}, in particular in comparison with models involving
TV and nonlocal means (NL) regularizers \cite{BCM05,GO07a,GO07,ST10}.

\section{Finite Discrete Shearlet Transform} \label{sec_dst}
We start with a brief introduction of the finite discrete shearlet transform to make the paper self-contained.
First we describe the continuous shearlet transform in $\R^2$ and on cones.
Then we provide a discretization of the shearlet transform on a finite domain
which is translation invariant and cannot be found in this way in the literature.
We prove that the discrete shearlets constitute a Parseval frame of the finite Euclidean space so that
the inversion of the shearlet transform  can be simply done by applying its adjoint transform.
For a more detailed description we refer to \cite{Hae11}.

\paragraph{Continuous Shearlets.}
Let the
\emph{parabolic scaling matrix} $A_a$
and the \emph{shear matrix} $S_s$
be defined by
\begin{equation*}
  A_a
  =
  \left(\begin{array}{cc} a & 0 \\ 0 & \sqrt{a}\end{array}\right),
  \ a\in\RR^+,
  \quad
  S_s
  =
  \left(\begin{array}{cc} 1 & s \\ 0 & 1\end{array}\right),
  \ s\in\RR.
\end{equation*}
Then the {\it continuous shearlets} $\psi_{a,s,t}$ emerge by dilation, shearing and translation of a function
$\psi \in L_2(\R)$ as
\begin{equation}\label{eq:psi_ast}
  \psi_{a,s,t}(x)
:=
  a^{-\frac{3}{4}}\psi(A_a^{-1}S_s^{-1}(x-t)),
\end{equation}
see \cite{KGL06,LWWW02}.
Using the {\it Fourier transform} $\mathcal{F}: L_2(\R^2) \rightarrow L_2(\R^2)$
given by
$$
  \mathcal{F}f(\omega)
=
  \hat{f}(\omega)
:=
  \int_{\R^2} f(t)e^{-2\pi i\langle\omega,t\rangle} dt,
$$
we obtain that
\begin{equation*}
  \hat{\psi}_{a,s,t}(\omega)
=
  a^{\frac{3}{4}} e^{-2\pi i\langle \omega,t \rangle} \hat{\psi}\left(a\omega_1,\sqrt{a}(s\omega_1+\omega_2)\right).
\end{equation*}
The \emph{continuous shearlet transform} $\mathcal{SH}_\psi (f)$ of a function $f\in L_2(\RR)$ is defined by
\begin{equation*}
  \mathcal{SH}_\psi (f)(a,s,t)
:=
  \langle f,\psi_{a,s,t}\rangle
  \; = \;
  \langle \hat{f},\hat{\psi}_{a,s,t} \rangle.
\end{equation*}
The shearlet transform is invertible if the function $\psi$ fulfills the admissibility property
$$\int_{\R^2} \frac{|\hat \psi(\omega_1,\omega_2)|^2}{|\omega_1|^2} \, d \omega_1 d\omega_2 < \infty.$$
This is in particular the case if
$$
\hat \psi(\omega) = \hat \psi_1(\omega_1) \hat \psi_2\left( \frac{\omega_2}{\omega_1} \right),
$$
where $\psi_1\colon\RR\to\RR$ is a wavelet and $\psi_2\colon\RR\to\RR$ a bump function.
Typical choices for these functions are
\begin{equation} \label{psi1_psi2}
  \hat{\psi}_1(\omega_1) := \sqrt{b^2(2\omega_1) + b^2(\omega_1)}
\quad {\rm and} \quad
\hat{\psi}_2(\omega_2) :=
\begin{cases}
  \sqrt{v(1+\omega_2)}&\text{for }\omega_2 \leq 0, \\
  \sqrt{v(1-\omega_2)}&\text{for }\omega_2 >0,
\end{cases}
\end{equation}
where
\begin{equation*}
 v(x) :=
\begin{cases}
  0&\text{for } x< 0, \\
  35x^4-84x^5+70x^6-20x^7&\text{for } 0\leq x \leq 1,\\
  1&\text{for } x>1
\end{cases}
\end{equation*}
and
\begin{equation*}
  b(x)
:=
  \begin{cases}
  \sin(\frac{\pi}{2}v(|x|-1))&\text{for }1\leq|x|\leq 2, \\
  \cos(\frac{\pi}{2}v(\frac{1}{2}|x|-1))&\text{for }2 < |x|\leq 4, \\
  0&\text{otherwise}
  \end{cases}
\end{equation*}
were introduced by Y. Meyer \cite{Me01,Ma08}.
The functions $\hat \psi_1$ and $\hat \psi_2$ are shown in Fig. \ref{fig:1}.
It is well-known that they have the following useful properties.
\begin{lemma} \label{lem:prop_psi1_psi2}
The functions ${\psi}_1$ and ${\psi}_2$ defined by \eqref{psi1_psi2} fulfill
\begin{align*}
  \sum_{j\geq 0}
  |\hat{\psi}_1(2^{-2j}\omega_1)|^2
&= 1\quad\text{for }|\omega_1|>1,\\
\sum_{k=-2^j}^{2^j}|\hat{\psi}_2(k+2^j\omega_2)|^2
&= 1\quad\text{for }|\omega_2|\leq 1,\ j\geq 0.
\end{align*}
\end{lemma}
%
\begin{figure}[htbp]
	\centering
	\includegraphics[width=0.4\textwidth]{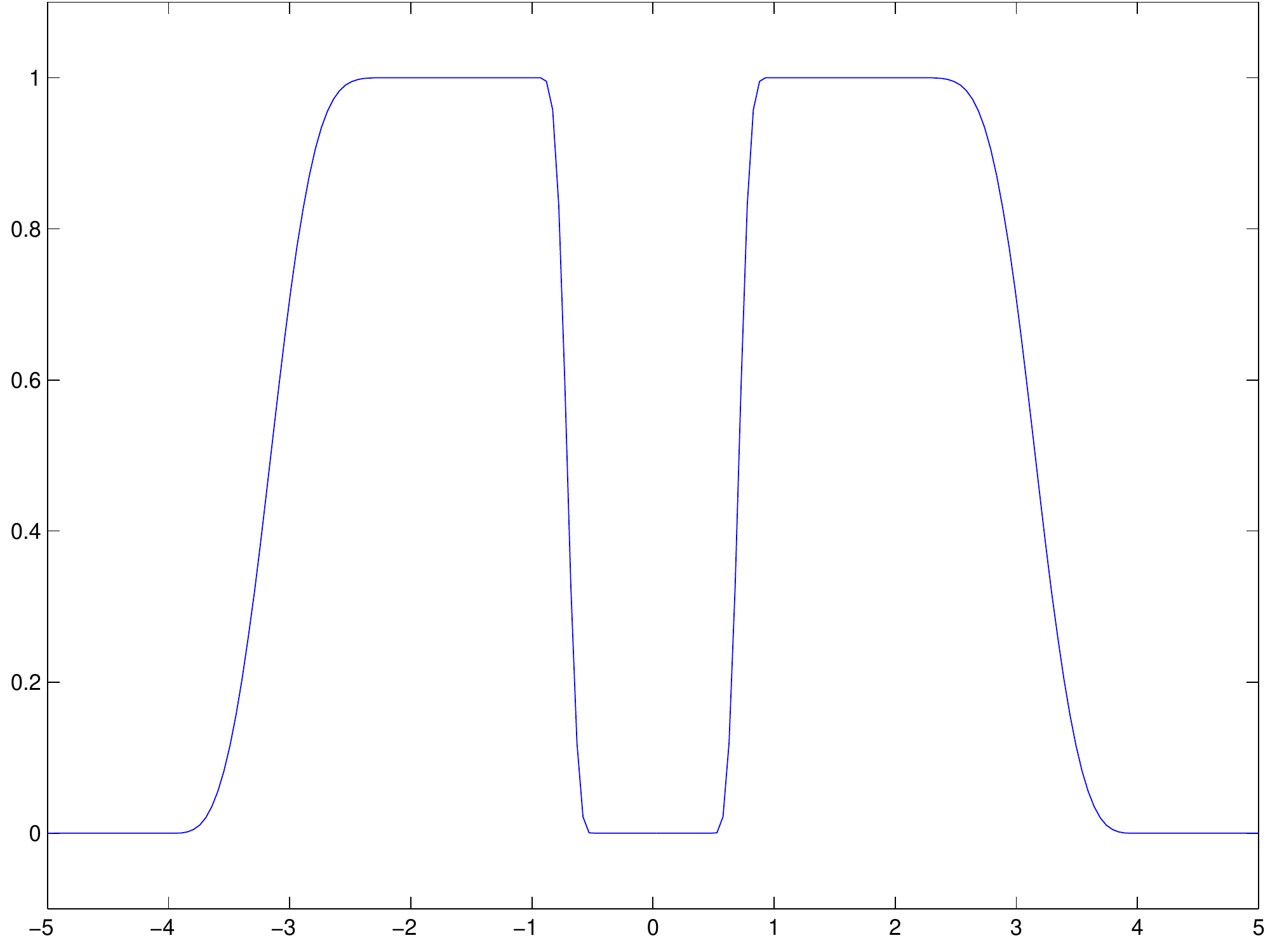}\hspace{1cm}
	\includegraphics[width=0.4\textwidth]{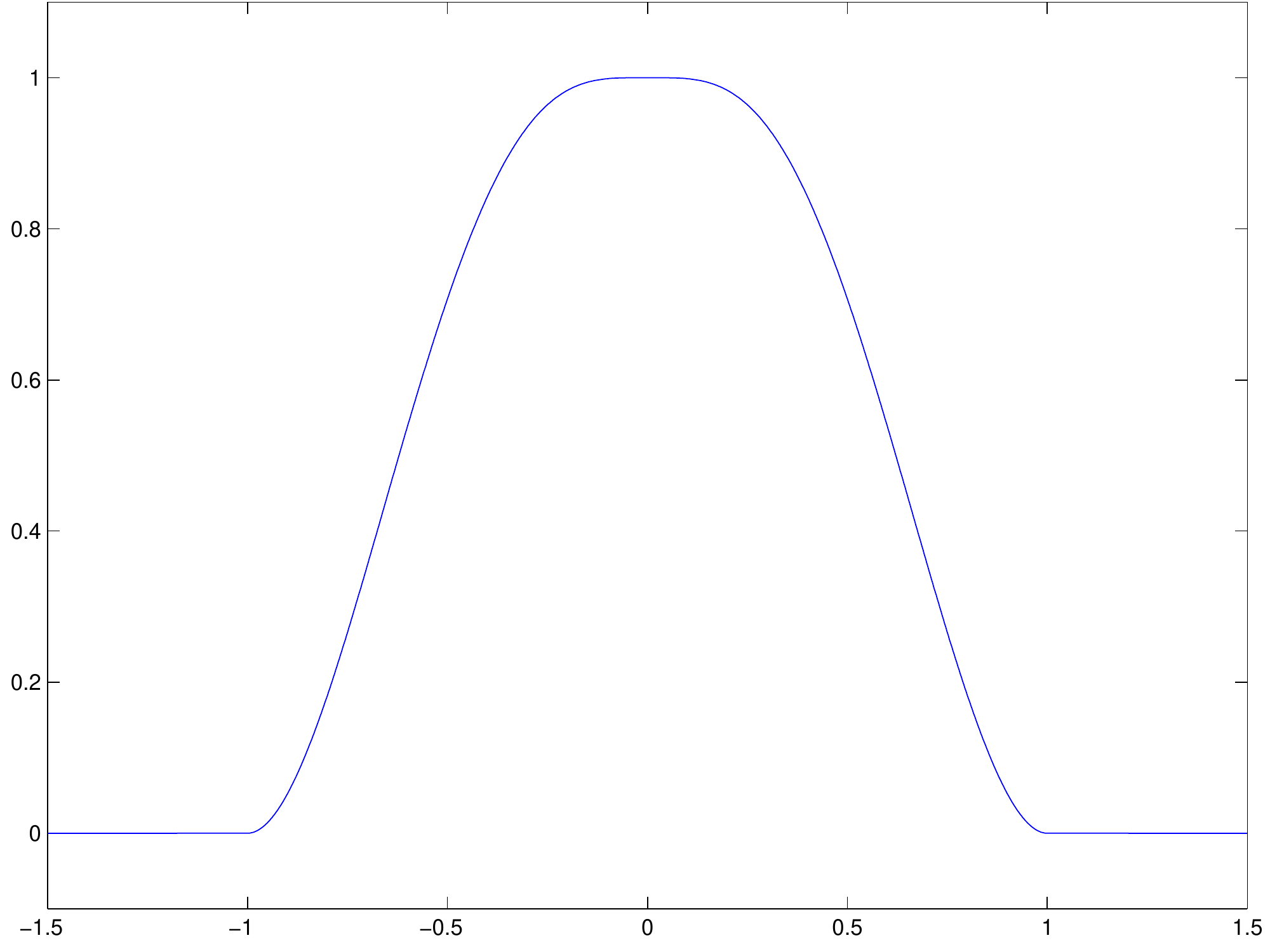}
	\caption{\label{fig:1} The functions $\hat \psi_1$ with ${\rm supp} \, \hat \psi_1 = [-4,-\frac12] \cup  [\frac12,4]$ (left)
and \newline $\hat \psi_2$ with ${\rm supp} \, \hat \psi_2 = [-1,1]$ (right).}
\end{figure}
\paragraph{Shearlets on Cones.}
To get a good directional selectivity both from the horizontal and vertical point of view
shearlets on the cone were introduced in \cite{KGL06}.
We define the \emph{horizontal} and the \emph{vertical cone} by
\begin{align*}
		\mathcal{C}^h
	&:=
		\{
			(\omega_1,\omega_2)\in\RR^2 : |\omega_1| \geq \frac{1}{2}, |\omega_2|<|\omega_1|
		\},\\
\mathcal{C}^v
	&:=
		\{
			(\omega_1,\omega_2)\in\RR^2 : |\omega_2| \geq \frac{1}{2}, |\omega_2|>|\omega_1|
		\},
\end{align*}
resp., and  the ``intersection'' (seam lines) of the two cones and the ``low frequency'' set by
\begin{align*}
			\mathcal{C}^\times
	&:=
		\{
			(\omega_1,\omega_2)\in\RR^2 : |\omega_1| \geq \frac{1}{2}, |\omega_2| \geq \frac{1}{2}, |\omega_1|=|\omega_2|
		\},
	\\
			\mathcal{C}^0
	&:=
		\{
			(\omega_1,\omega_2)\in\RR^2 : |\omega_1| < 1, |\omega_2| < 1
		\},
\end{align*}
resp., see Fig. \ref{fig:cones}. Altogether
$
 \RR^2
=
 \mathcal{C}^h \cup
 \mathcal{C}^v \cup
 \mathcal{C}^\times \cup
 \mathcal{C}^0
$
with an overlapping domain
$\cone^\square:=(-1,1)^2\setminus(-\frac{1}{2},\frac{1}{2})^2$.
\begin{figure}[htbp]
	\centering
	\includegraphics[width=0.45\textwidth]{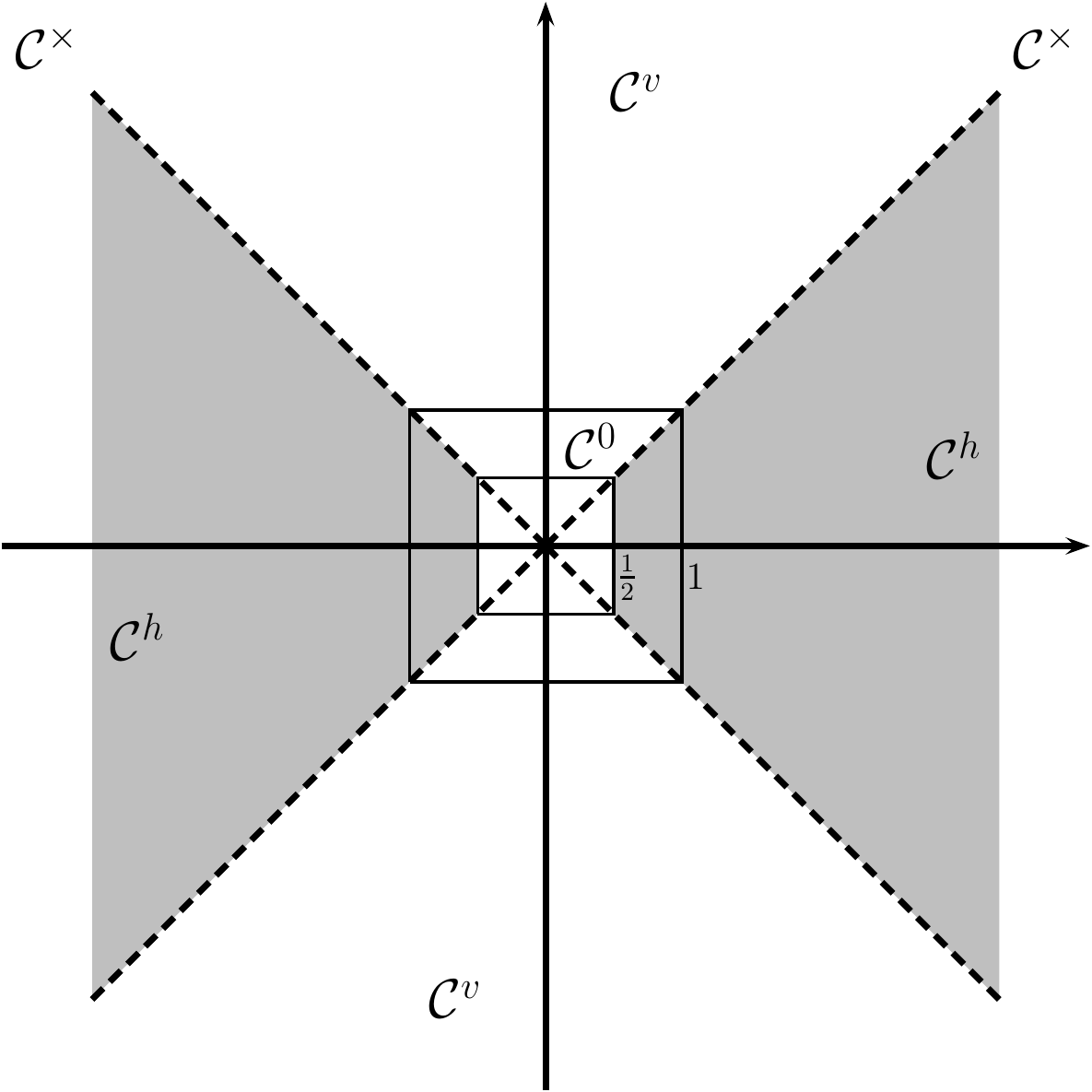}
	\caption{\label{fig:cones} The sets $\mathcal{C}^h$, $\mathcal{C}^v$, $\mathcal{C}^\times$ and $\mathcal{C}^0$.}
	\end{figure}
For each set $\cone^\kappa$, $\kappa \in \{h,v,\times\}$ we define a characteristic function $\chi_{\cone^\kappa}$
which is equal to 1 for $\omega\in\cone^\kappa$ and 0 otherwise.
The {\it shearlets on the cone} are defined by
\begin{equation*}
	\hat{\psi}^h(\omega_1,\omega_2)
:=
	\hat{\psi}_1(\omega_1)
	\hat{\psi}_2\left(\frac{\omega_2}{\omega_1}\right)
	\chi_{\cone^h}
\quad {\rm and} \quad
	\hat{\psi}^v(\omega_1,\omega_2)
:=
	\hat{\psi}_1(\omega_2)
	\hat{\psi}_2\left(\frac{\omega_1}{\omega_2}\right)
	\chi_{\cone^v}.
\end{equation*}
Fig. \ref{fig:shearletFourierTime} shows shearlets both in the time and frequency domain.
Note that for $|s| \leq 1 - \sqrt{a}$ we have $\supp\hat{\psi}_{a,s,t}\subseteq\cone^h$ and for $1-\sqrt{a}<|s|<1+\sqrt{a}$
parts of $\supp\hat{\psi}_{a,s,t}$ are also in $\cone^v$
which are cut off. For $|s|>1+\sqrt{a}$ the whole shearlets are set to zero by the characteristic function.
For $(\omega_1,\omega_2)\in\cone^\times$, i.e. $|\omega_1|=|\omega_2|$, we define
$
	\hat{\psi}^\times(\omega_1,\omega_2)
	:=
	\hat{\psi}(\omega_1,\omega_2)\chi_{\cone^\times}.
$
Finally, we use the scaling function $\phi$ defined by
\begin{align*}
  \hat{\phi}(\omega_1,\omega_2)
&:=
  \begin{cases}
    \varphi(\omega_1) &\text{for } |\omega_1| < 1, |\omega_2|\leq|\omega_1|, \\
    \varphi(\omega_2) &\text{for } |\omega_2| < 1, |\omega_1|<|\omega_2| \\
  \end{cases}
\end{align*}
with
\begin{equation*} \label{triangle}
\varphi(\omega) :=
\begin{cases}
  1 &\text{for } |\omega|\leq\frac{1}{2}, \\
  \cos(\frac{\pi}{2}v(2|\omega|-1))&\text{for }\frac{1}{2}<|\omega|<1, \\
  0 &\text{otherwise}
\end{cases}
\end{equation*}
and its translates $\phi_{t}(x) = \phi(x-t)$, i.e.,
$\hat{\phi_t}(\omega) = e^{-2\pi i \langle t, \omega \rangle} \hat{\phi}(\omega)$ on ${\cal C}^0$.
Note that this  scaling function $\hat{\phi}$ (respectively $\varphi$)
matches perfectly with $\hat{\psi}_1$, more precisely, since
\begin{equation*}
 \sum_{j\geq 0}
  |\hat{\psi}_1(2^{-2j}\omega)|^2
=
\begin{cases}
 0                                                 &\text{for }|\omega|\leq\frac{1}{2}, \\
 \sin^2\left(\frac{\pi}{2}v(2\omega-1)\right) &\text{for }\frac{1}{2}<|\omega|<1,  \\
 1                                                 &\text{for }|\omega|\geq 1
\end{cases}
\end{equation*}
we obtain
\begin{equation}\label{eq:OverlapPsi1_2}
  |\hat{\psi}_1(\omega)|^2+|\varphi(\omega)|^2
=
   1 \quad {\rm for} \; |\omega|\in \left[\tfrac{1}{2},1\right].
\end{equation}
\paragraph{Finite Discrete Shearlets.}
In the following, we consider digital images as functions sampled on the grid
$\frac{1}{N}{\cal I}:= \frac{1}{N} \{(m_1,m_2): \; m_i = 0,\ldots,N-1, \, i=1,2 \}$,
where we assume quadratic images for simplicity.
Further, let $j_0:=\lfloor \frac{1}{2}\log_2 N\rfloor$ be the number of considered scales.
To obtain a discrete shearlet transform, we
discretize the scaling, shear and translation parameters as
\begin{align*}
  a_j
&:=
  2^{-2j} = \frac{1}{4^j},\quad j= 0,\ldots,j_0-1,
  \\
  s_{j,k}
&:=
  k2^{-j},\quad -2^j\leq k\leq 2^j,
  \\
  t_m
&:=
  \frac{m}{N},\quad m \in {\cal I}.
\end{align*}
With these notations our shearlets  becomes
$
  \psi_{j,k,m}(x)
:=
  \psi_{a_j,s_{j,k},t_{m}} (x)
=
  \psi(A_{a_j}^{-1}S_{s_{j,k}}^{-1}(x-t_{m}))
$.
Observe that compared to the continuous shearlets defined in (\ref{eq:psi_ast}) we omit the factor $a^{-\frac{3}{4}}$.
In the Fourier domain we obtain
\begin{equation*}
 \hat{\psi}_{j,k,m}(\omega)
=
  \hat{\psi} (A_{a_j}^\tT S_{s_{j,k}}^\tT\omega)
  e^{-2\pi i \langle \omega, t_m \rangle}
  =
  \hat{\psi}_1
  \left(
	4^{-j}\omega_1
  \right)
  \hat{\psi}_2
  \left(
	2^j\frac{\omega_2}{\omega_1} + k
  \right)
  e^{-2\pi i \langle \omega, m \rangle/N}, \quad \omega \in \Omega,
\end{equation*}
where
$\Omega
  :=\left\{ (\omega_1,\omega_2):\, \omega_i =
	  -\left\lfloor\frac{N}{2}\right\rfloor, \ldots , \left\lceil\frac{N}{2}\right\rceil -1, \, i=1,2
	\right\}.
$
\begin{figure}[htbp]
    \centering
    \subfigure[Shearlet in Fourier domain \newline for $a=\frac{1}{4}$ and $s=-\frac{1}{2}$ ]{\label{fig:shearletFourier}
\includegraphics[width=0.4\textwidth]{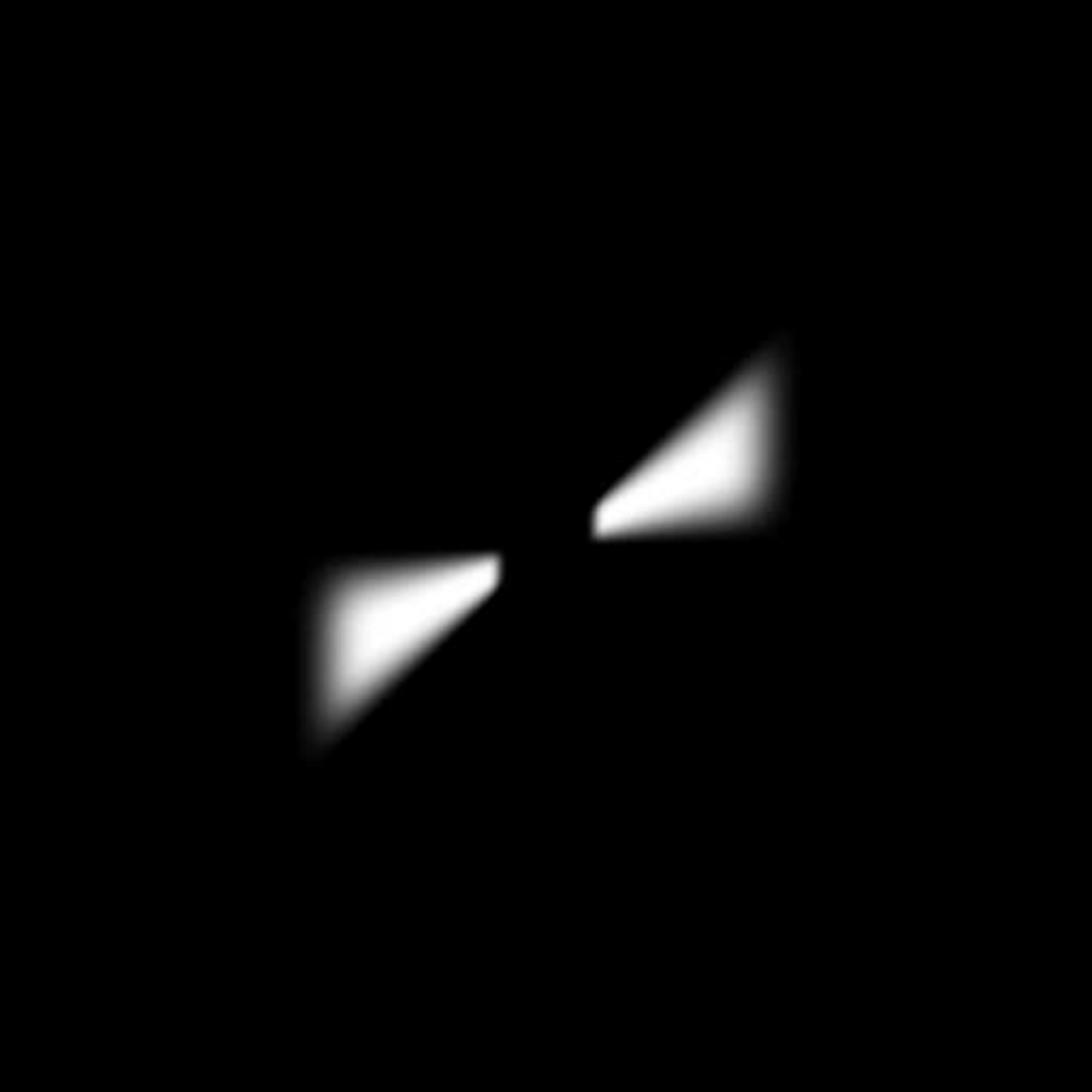}}
\hspace{1cm}
    \subfigure[Same shearlet in time domain (zoomed)]{\label{fig:shearletTime}
\includegraphics[width=0.4\textwidth]{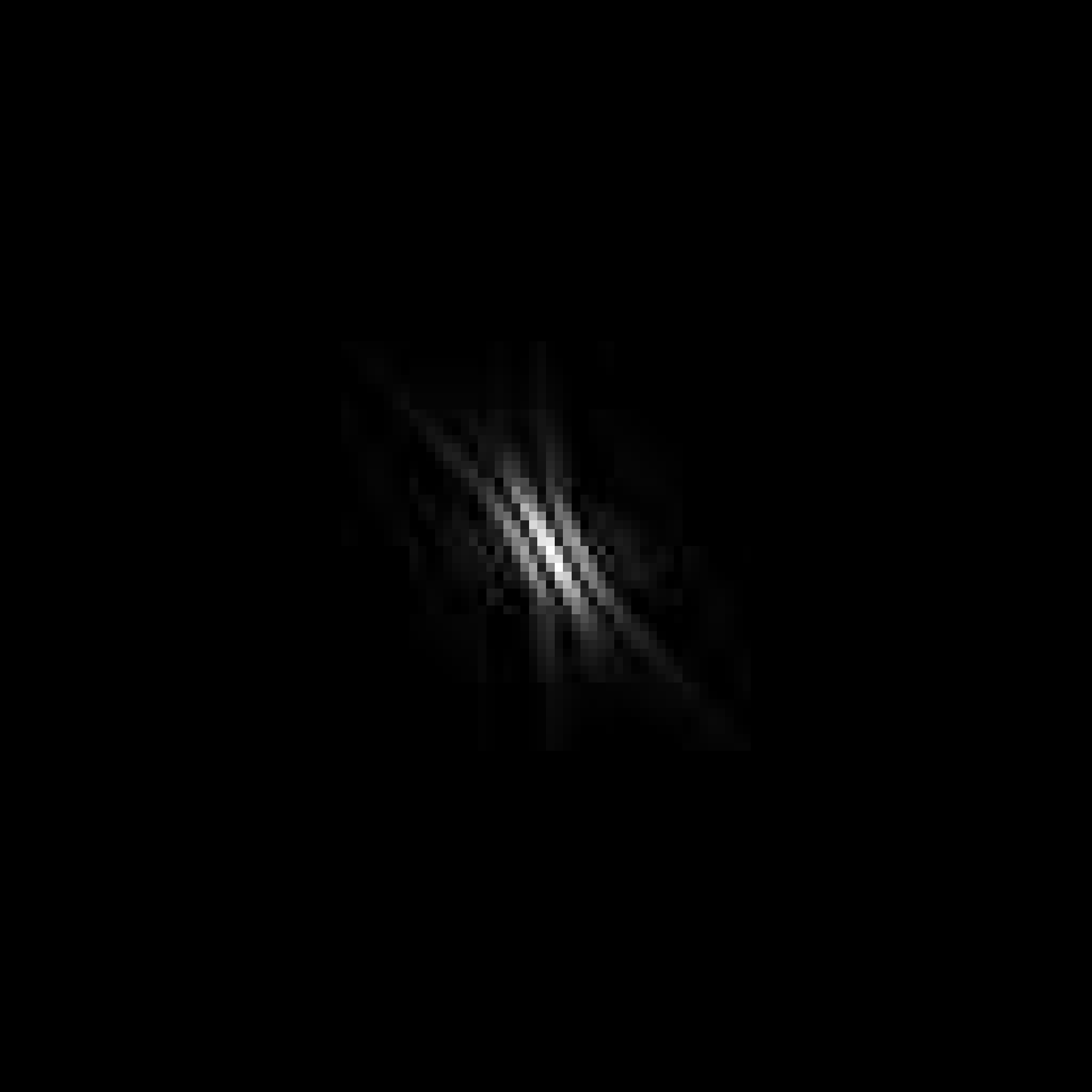}}
    \caption{Shearlet in Fourier and time domain.}
    \label{fig:shearletFourierTime}
\end{figure}
By definition we have $a\leq 1$ and $|s|\leq 1$.
Therefore we see that we have a cut off due to the cone boundaries only for $|k| = 2^j$ where $|s| = 1$.
For both cones we have for $|s|=1$ two ``half'' shearlets with a gap at the seam line.
None of the shearlets are defined on the seam line
$\cone^\times$. To obtain ``full'' shearlets at the seam lines we ``glue'' the three parts together,
thus, we define for $|k| = 2^j$ a sum of shearlets
$$
  \hat{\psi}^{h\times v}_{j,k,m}
:=
  \hat{\psi}^h_{j,k,m}
  +
  \hat{\psi}^v_{j,k,m}
  +
  \hat{\psi}^\times_{j,k,m}.
$$
We define the \emph{discrete shearlet transform} as
\begin{equation*}\label{eq:fullShearletTransform}
  \mathcal{SH}(f)(\kappa,j,k,m)
:=
  \begin{cases}
    \langle f,\phi_m \rangle 		&\text{for }\kappa = 0,  \\
    \langle f,\psi^\kappa_{j,k,m} \rangle 	&\text{for }\kappa \in \{h,v\},  \\
    \langle f,\psi^{h\times v}_{j,k,m} \rangle 	&\text{for }\kappa = \times, |k| =  2^j.
  \end{cases}
\end{equation*}
where $j= 0,\ldots,j_0-1$, $-2^j+1\leq k \leq 2^j-1$ and $m\in {\cal I}$ if not stated in another way.
The shearlet transform can be efficiently realized by applying the \fftn\ and its inverse \ifftn\
which compute the following discrete Fourier transforms
with ${\cal O}(N^2 \log N)$ arithmetic operations:
\begin{align*}
  \hat{f} (\omega)
&=
  \sum_{m \in {\cal I}}  f(m)   e^{-2\pi i \langle \omega, m \rangle/N}, \quad \omega \in \Omega,
  \\
  f(m)
&=
  \frac{1}{N^2}
  \sum_{\omega \in \Omega}    \hat f(\omega)
  e^{2\pi i \langle \omega, m \rangle/N}, \quad m \in {\cal I}.
\end{align*}
We have the  Plancherel formula
$$
 \langle f, g \rangle = \frac{1}{N^2} \langle \hat f, \hat g \rangle.
$$
Thus, the discrete shearlet transform can be computed for $\kappa=h$ as follows:
\begin{align}
  \mathcal{SH}(f)(h,j,k,m) \label{eq:wichtig}
&=
  \langle f, \psi_{j,k,m}^h \rangle
=
  \frac{1}{N^2}
  \langle \hat{f}, \hat{\psi}_{j,k,m}^h \rangle
  \\
&=
  \frac{1}{N^2}
  \sum_{\omega\in\Omega}
    \underbrace{\hat{\psi}
  (4^{-j}\omega_1,4^{-j}k\omega_1+2^{-j}\omega_2)
  \hat{f}(\omega_1,\omega_2)}_{\hat{g}_{j,k}(\omega)}\, e^{2\pi i\langle \omega,m \rangle/N} \nonumber
\\
&=\ifftn(\hat{g}_{j,k})\nonumber
\end{align}
and with the corresponding modifications for $\kappa \in \{v,\times,0 \}$.

In view of the inverse shearlet transform we prove that our discrete shearlets
constitute a Parseval frame of the finite Euclidean space $L_2({\cal I})$.
Recall that for a Hilbert space $\mathcal{H}$ a sequence $\{u_j:j\in {\cal J}\}$
is a \emph{frame} if and only if there exist constants $0 < A\leq B <\infty$ such that
$$
  A\|f\|^2
\leq
  \sum_{j\in {\cal J}}
  |\langle
	f, u_j
  \rangle|^2
\leq
  B \|f\|^2
\quad\text{for all }
f\in\mathcal{H}.
$$
The frame is called \emph{tight} if $A=B$ and a \emph{Parseval} frame if $A = B = 1$.
For Parseval frames the reconstruction formula
$$
  f =
  \sum_{j\in\mathcal{J}}
  \langle f,u_j\rangle u_j
\quad\text{for all }f\in\mathcal{H}
$$
holds true, see \cite{Ch03,Ma08}. In the $n$-dimensional Euclidean space we can
arrange the frame elements $u_j$, $j = 1,\ldots,\tilde n \ge n$ as rows of a matrix $U$.
Then we have indeed a frame if $U$ has full rank and a Parseval frame if and only if $U^\tT U = I_n$.
Note that $U U^\tT = I_{\tilde n}$ is only true if the frame is an orthonormal basis.
The Parseval frame transform and its inverse read
\begin{equation}\label{eq:trafo}
(\langle f,u_j \rangle)_{j =1}^{\tilde n} = U f
\quad {\rm and} \quad
f = U^\tT (\langle f,u_j \rangle)_{j =1}^{\tilde n}.
\end{equation}
By the following theorem our shearlets provide such a convenient system.
%
\begin{theorem}\label{theorem:discreteShearletFrame}
The discrete shearlet system
\begin{equation*}
\{
\psi^h_{j,k,m}, \,\psi^v_{j,k,m}, \, \psi^{h\times v}_{j,\pm 2^j,m}: j=0,\ldots,j_0-1, -2^{j}+1\leq k\leq 2^j-1, m \in {\cal I}
\}
\; \cup \;
    \{
	  \phi_{m} : m\in{\cal I}
    \}
\end{equation*}
provides a Parseval frame for $L_2({\cal I})$.
\end{theorem}
%
\begin{proof}
We have to show that
\begin{equation*}
 \|f\|^2
=
  \sum_{\kappa\in\{h,v\}} \sum_{j=0}^{j_0-1} \sum_{k=-2^j+1}^{2^j-1} \sum_{m\in {\cal I}}
  |\langle f,\psi_{j,k,m}^\kappa \rangle|^2
+
  \sum_{j=0}^{j_0-1} \sum_{k=\pm 2^j} \sum_{m\in {\cal I}}
  |\langle f,\psi_{j,k,m}^{h\times  v} \rangle|^2
+
  \sum_{m\in {\cal I}}
  |\langle f,\phi_{m} \rangle|^2 = : C.
\end{equation*}
By \eqref{eq:wichtig} we know that
$$
	\langle f,\psi_{j,k,m}^h \rangle
 =
	\frac{1}{N^2}
	\sum_{\omega\in\Omega} \hat g_{j,k} (\omega)e^{2\pi i\langle \omega,m \rangle/N}
  =
	g_{j,k}(m),
$$
so that by Parseval's formula
$$
\sum_{m\in{\cal I}}
  |\langle
	f, \psi_{j,k,m}^h
  \rangle|^2
=
  \sum_{m\in{\cal I}}
  |g_{j,k}(m)|^2
=
  \|g_{j,k}\|^2
=
\frac{1}{N^2} \|\hat g_{j,k}\|^2.
$$
Analogous computations can be done for the vertical cone,
the seam-line part and the  low-pass part, see \cite{Hae11}.
Summing up the different pieces we obtain
\begin{equation*}
\begin{split}
  C
&=
  \frac{1}{N^2} \left(
  \sum_{j= 0}^{j_0-1}
  \sum_{k=-2^j+1}^{2^j-1}
  \sum_{\omega\in\Omega}
  |  \hat{\psi}   (4^{-j}\omega_1,4^{-j}k\omega_1+2^{-j}\omega_2)  |^2
  |  \hat{f}(\omega_1,\omega_2)  |^2\right.
  \\
&\quad +
  \sum_{j= 0}^{j_0-1}
  \sum_{k=-2^j+1}^{2^j-1}
  \sum_{\omega\in\Omega}
  |\hat{\psi} (4^{-j}\omega_2,4^{-j}k\omega_2+2^{-j}\omega_1)|^2
  |\hat{f}(\omega_1,\omega_2)|^2
  \\
&\quad +
  \sum_{j= 0}^{j_0-1} \sum_{k=\pm 2^j}
  \left(
  \sum_{\omega\in\Omega}
  |\hat{\psi} (4^{-j}\omega_1,4^{-j}k\omega_1+2^{-j}\omega_2)|^2
  |
	\hat{f}(\omega_1,\omega_2)
  |^2
  \chi_{\cone^h}
  \right.
  \\
&\quad +
  \sum_{\omega\in\Omega}
  |
	\hat{\psi}
	(4^{-j}\omega_2,4^{-j}k\omega_2+2^{-j}\omega_1)
  |^2
  |
	\hat{f}(\omega_1,\omega_2)
  |^2
  \chi_{\cone^v}
  \\
&\quad +
  \left.
  \sum_{\omega\in\Omega}
  |
	\hat{\psi}
	(4^{-j}\omega_1,4^{-j}k\omega_1+2^{-j}\omega_2)
  |^2
  |
	\hat{f}(\omega_1,\omega_2)
  |^2
  \chi_{\cone^\times}
  \right)
  \\
&\quad + \left.
  \sum_{\omega\in\Omega}
  |
	\hat{\phi}(\omega_1,\omega_2)
  |^2
  |
	\hat{f}(\omega_1,\omega_2)
  |^2
\right)
\end{split}
\end{equation*}
which can be further simplified by Lemma \ref{lem:prop_psi1_psi2} as
\begin{equation*}
	\begin{split}
C &= \frac{1}{N^2}
\biggl(
	\sum_{\omega\in\cone^h}
	|
	  \hat{f}(\omega_1,\omega_2)
	|^2
	\underbrace{
	  \sum_{j= 0}^{j_0-1}
	  |
		\hat{\psi}_1(4^{-j}\omega_1)
	  |^2
	}_{\equiv 1 \text{ for } |\omega_1|\geq 1}
	\underbrace{
	\sum_{k=-2^j}^{2^j}
	  |
		\hat{\psi}_2(2^j\frac{\omega_2}{\omega_1} + k)
	  |^2
	}_{
	  \equiv 1 }\biggr.
  \\
&\quad +
	\sum_{\omega\in\cone^v}
	|
	  \hat{f}(\omega_1,\omega_2)
	|^2
	\underbrace{
	  \sum_{j= 0}^{j_0-1}
	  |
		\hat{\psi}_1(4^{-j}\omega_2)
	  |^2
	}_{
	  \equiv 1\text{ for } |\omega_2| \geq 1
	}
	\underbrace{
	\sum_{k=-2^j}^{2^j}
	  |
		\hat{\psi}_2(2^j\frac{\omega_1}{\omega_2} + k)
	  |^2
	}_{
	  \equiv 1
	}
	\\
&\quad +
	\sum_{\omega\in\Omega}
	\sum_{j=0}^{j_0-1}
	 |\hat{f}(\omega_1,\omega_2)|^2
	|
	\underbrace{
	\hat{\psi}
	  (4^{-j}\omega_1,0)
	}_{
	\equiv 1
	}
	|^2
	\chi_{\cone^\times}
  \\
&\quad +\biggl.
  \sum_{\omega\in\Omega}
  |
  \underbrace{
	\hat{\phi}(\omega_1,\omega_2)
  }_{
	\equiv 1\text{ for }\omega\in[-\frac{1}{2},\frac{1}{2}]^2
  }
  |^2
  |
  \hat{f}(\omega_1,\omega_2)
  |^2
\biggr).
	\end{split}
\end{equation*}
Finally we obtain by \eqref{eq:OverlapPsi1_2} that
\begin{align*}
  C
&=
  \frac{1}{N^2} \Bigl(
  \sum_{\omega\in\Omega\setminus\cone^\square}
  |\hat{f}(\omega_1,\omega_2)  |^2
\; + \;
  \sum_{\omega\in\cone^\square}
  |	\hat{f}(\omega_1,\omega_2)   |^2
\bigl(
	  |\hat{\psi}_1(\omega_1)|^2 +  |\hat{\psi}_1(\omega_2)|^2 + |\hat{\phi}(\omega_1,\omega_2)|^2
  \bigr)
\Bigr)
  \\
&=
  \frac{1}{N^2}
  \sum_{\omega\in\Omega}
  |
	\hat{f}(\omega_1,\omega_2)
  |^2
\; = \;
  \|f\|^2.
\end{align*}
and we are done.
\end{proof}
Having a Parseval frame the {\it inverse shearlet transform} reads
\begin{equation*}
  f
=
  \sum_{\kappa\in\{h,v\}}\sum_{j= 0}^{j_0-1} \sum_{k=-2^j+1}^{2^j-1}\sum_{m\in{\cal I}}
  \langle f,\psi_{j,k,m}^\kappa \rangle \psi_{j,k,m}^\kappa
+
  \sum_{j= 0}^{j_0-1}\sum_{k=\pm 2^j}\sum_{m\in{\cal I}}
  \langle f,\psi_{j,k,m}^{h\times  v} \rangle \psi_{j,k,m}^{h\times  v}
+
  \sum_{m\in{\cal I}}
  \langle f,\phi_{m} \rangle \phi_{m}.
\end{equation*}
The actual computation of $f$ from given coefficients $c(\kappa,j,k,m) := \langle f,\psi_{j,k,m}^\kappa \rangle$
is done in the Fourier domain. Due to the linearity of the Fourier transform we get,
e.g., for the horizontal cone
\begin{align*}
  \hat{f}(\omega)\chi_{\cone^h}
&=
  \sum_{j= 0}^{j_0-1}
  \sum_{k=-2^j+1}^{2^j-1}
  \sum_{m\in{\cal I}}
  \langle
  f,
  \psi^h_{j,k,m}
  \rangle
  \hat{\psi}^h_{j,k,m}(\omega)
  \\
&=
  \sum_{j= 0}^{j_0-1}
  \sum_{k=-2^j+1}^{2^j-1}
 \underbrace{ \sum_{m\in{\cal I}}
  c(h,j,k,m)
  e^{-2\pi i \langle \omega,m \rangle}
}_{\fftn(c(h,j,k,\cdot))(\omega_1,\omega_2)}
  \hat{\psi}
  (4^{-j}\omega_1,4^jk\omega_1+2^{-j}\omega_2).
\end{align*}
The inner sum can be interpreted as a two-dimensional discrete Fourier transform and can be computed via a FFT.
Hence,
$\hat{f}$ can be computed by simple multiplications of these Fourier-transformed shearlet coefficients
with the dilated and sheared spectra of $\psi$ and afterwards summing over all  scales $j$ and all shears $k$.
The spectra are the same as for the transform itself and can be reused.
Finally we get $f$ itself by an iFFT of $\hat{f}$.
For implementation details we refer to \cite{Hae11}.

\begin{remark} \label{rem:other}
1.
We are aware of our larger oversampling factor in comparison with, e.g., \emph{ShearLab}.
Having 4 scales we obtain 61 images of the same size as the original image.
But since shearlets are designed to detect edges in images we like to avoid any down-sampling and keep translation invariance.
A possibility to reduce the memory usage would be to use the compact support of the shearlets in the frequency domain and only compute them on a ``relevant'' region.
But we then would also have to store the position and size of each region what decreases the memory savings and would make the implementation a lot more complicated.

2.
In \cite{SKZ11} and the respective implementation \emph{ShearLab} a pseudo-polar Fourier-transform is used to implement a discrete (or digital) shearlet transform.
For the scale $a$ and the shear $s$ the same discretization is used.
But for the translation $t$ the authors set $t_{j,k,m}:=A_{a_j}S_{s_{j,k}}m$.
Thus, their discrete shearlet becomes
\begin{equation*}
  \hat{\widetilde{\psi}}_{j,k,m}(\omega)
=
  \hat{\psi}(A_{a_j}S_{s_{j,k}}^\tT\omega)e^{-2\pi i\langle \omega, A_{a_j}S_{s_{j,k}}m \rangle}
=
  \hat{\psi}(A_{a_j}S_{s_{j,k}}^\tT\omega)e^{-2\pi i\langle S_{s_{j,k}}^\tT A_{a_j}\omega, m \rangle}
  .
\end{equation*}
Since the operation $S_{s_{j,k}}^\tT A_{a_j}\omega$ would destroy the pseudo-polar grid a ``slight'' adjustment is made and the exponential term is replaced by
\begin{equation*}
	e^{-2\pi i\langle (\theta\circ S_{s_{j,k}}^{-\tT})S_{s_{j,k}}^\tT A_{a_j}\omega, m \rangle}
\end{equation*}
with $\theta\colon \RR\setminus\{0\}\times\RR \to \RR\times\RR$ and $\theta(x,y) = (x,\frac{y}{x})$ such that
\begin{equation*}
	e^{-2\pi i\langle (\theta\circ S_{s_{j,k}}^{-\tT})S_{s_{j,k}}^\tT A_{a_j}\omega, m \rangle}
=
	e^{-2\pi i\left\langle \left(a_j\omega_1, \sqrt{a_j}\frac{\omega_2}{\omega_1}\right), m \right\rangle}.
\end{equation*}
With this adjustment the last step of the shearlet transform can be obtained with a standard inverse fast Fourier transform
(similar as in our implementation).
Unfortunately this is no longer related to translations of the shearlets in the time domain.

3.
Having a look at the Preprint of this paper,  D. Labate  made us aware of his Preprint \cite{GL11} where he and K. Guo proposed an interesting, new construction of smooth diagonal shearlets in contrast to our
continuous diagonal shearlets $\psi^{h\times v}$. We implemented both constructions and
as the smooth approach is promising for theoretical purposes the differences in our applications were negligible.
\end{remark}

\section{Shearlet Regularization in Image Segmentation} \label{sec_segm_shearlets}
In this section, we incorporate shearlets into a convex multi-label segmentation model.
We start by explaining the convex relaxation model in matrix-vector notation. The approach is based on
\cite{LBS09,LS10,SS11} and is closely related to \cite{BYT09,EG11,JKS07,KSY11,PCCB09,SKC10,ZGFN08}.

\paragraph{Segmentation Model.}
We segment an image by assigning different labels to different areas of the image based on the gray or RGB value of the pixels.
We assume that the respective values are given in a codebook.
In the following, we restrict ourselves to gray value images and will comment the slight adjustments for RGB images later.
For a simple matrix-vector notation we assume the image to be column-wise reshaped as a vector.
Thus, having an image $f\in\RR^{N\times N}$
we obtain the vectorized image $\text{vec}(f)\in\RR^{N^2}$ where we retain the notation $f$ for both the original and the reshaped image.
Let $q\in\NN$ be the number of clusters
which will be labeled by labels in $\mathcal{C}:=\{1,\ldots,q\}$.

For a given image $f\in\RR^{N^2}$ and a codebook $c=(c_1,\ldots,c_q)$ we want to find a labeling vector $u^\star\in\{0,1\}^{qN^2}$
that provides a ``good'' segmentation of $f$. For the interpretation it is more convenient to take $u^\star\in\{0,1\}^{N\times N\times q}$
as $q$ layers of an image in $\{0,1\}^{N\times N}$.
Usually for each pixel only the entry in one layer is $1$ and the entries in the other layers are $0$.
Consequently the pixel gets the label according to the index of the respective layer.
As a relaxation we allow non-zero entries in every layer but restrict the sum over all layers for each pixel to be equal to $1$,
i.e., $u^\star\in [0,1]^{qN^2}$ and $\sum_{k=0}^{q-1} u^\star[r+kN^2] = 1$
for all $1\leq r \leq N^2$. The respective label is chosen according to the index of the largest entry for each pixel.

The following variational approach finds $u^\star$ as the minimizer of the functional
\begin{equation*}
  F(u) = \langle u,s \rangle + \lambda \Psi(u)\quad\text{subject to }u\in C
\end{equation*}
where
\begin{equation*}
  C := \{
  x\in[0,1]^{qN^2}
  :
  \sum_{k=0}^{q-1}
  x[r+kN^2]
  = 1\text{ for all }
  1\leq r \leq N^2
  \}.
\end{equation*}
The constraint guarantees that the sum over all layers for each pixel is equal to $1$.
We will refer to the first term as the \emph{data term} and to the second as
 \emph{regularizer} and $\lambda > 0$ as the
\emph{regularization parameter}.

The relation to the given data $f$ is contained in $s\in\RR^{qN^2}$ which penalizes a certain distance between $f$ and the codebook $c$.
In each layer $i$ of $s$ we subtract the codebook value $c_i$ from $f$ - in formulas:
\begin{equation}\label{eq:computationOfs}
	s[r+(k-1)N^2]
=
  \|f[r]-c[k]\|_p^p
\quad\text{for }
  1\leq r\leq N^2,\
  1\leq k\leq q.
\end{equation}
Thus, for big values in $s$ the data term gets small for small values in $u$. On the other hand for small values in $s$
we can chose values in $u$ close to $1$ what is needed for the sum to be equal to $1$.
The trivial solution $u=0$ is excluded due to the constraint $u\in C$. The regularization term is chosen to provide ``smooth'' layers.

With the regularization parameter one can weight
between data exactness and smoothness.
A meanwhile frequent choice for the regularizer
is the discrete \emph{TV-functional} in its isotropic or anisotropic form.
Note that for the anisotropic choice graph cut algorithms
can be efficiently applied,
see, e.g., \cite{BVZ01,BT09} and \cite{YBBT11} for a convex relaxation approach through continuous max-flow.
This is also true for penalizers involving anisotropic variants of NL-means.
While the TV-functional is well-suited for cartoon-like parts,
NL-means are a powerful tool for restoring textured regions in images \cite{BCM05,GO07,GO07a}.
We will use the TV-functional and NL-means regularizers in an isotropic form for numerical comparisons
in Section \ref{sec:num_results}.

In this paper, we propose to minimize the functional
\begin{equation*}
  F(u) = \langle u,s\rangle + \|\Lambda (I_q \otimes \mathcal{S}) u \|_1 + \iota_C(u),
\end{equation*}
where $I_q$ is the identity matrix in $\RR^{q\times q}$ and $\mathcal{S}$ is the shearlet-transform applied to a vector in $\RR^{N^2}$ in the sense of
\eqref{eq:trafo}. In particular we have that
\begin{equation} \label{eq:important}
\mathcal{S}^\tT \mathcal{S} = I_{N^2}.
\end{equation}
Finally, $I_q \otimes \mathcal{S}$ simply
applies this shearlet-transform to each ``layer'' of $u$ separately.
The indicator function $\iota_C$ is defined by
\begin{equation*}
  \iota_C(u)
  :=
  \begin{cases}
    0,&u\in C \\
    \infty, &u\not\in C.
  \end{cases}
\end{equation*}
Observe that our ``regularization parameter'' $\Lambda$ is a diagonal matrix since we want to apply
different parameters to different scales of the shearlet transform. With these notations
we can choose different parameters for each scale and shear and even for each layer. In the following we will use the same parameter
for each shear but different parameters for each scale.
Thus, we simply write $\Lambda = (\lambda_0,\lambda_1,\ldots,\lambda_{j_0})$
where $\lambda_0$ is the parameter for the ``low frequency'' and
$\lambda_1,\ldots,\lambda_{j_0}$ are the parameters for the different scales.

Summing up, finding an appropriate labeling vector $u^\star$ for the image $f$ is equivalent to solving the problem
\begin{equation}\label{eq:problemSegmentation}
u^\star = \argmin_{u\in\RR^{qN^2}} \{\langle u,s \rangle +  \| \Lambda(I_q \otimes \mathcal{S}) u \|_1 + \iota_C(u) \}.
\end{equation}

\paragraph{Alternating Direction Method of Multipliers.}
To solve the problem in (\ref{eq:problemSegmentation}) we want to use the ADMM which applies in general to
\begin{equation*}\label{eq:ADMMproblem}
\argmin_{x\in\RR^m, y\in\RR^n}\{G_1(x) + G_2(y)\}\text{ subject to } Ax = y
\end{equation*}
with proper, closed, convex functions $G_1:\RR^m\to\RR\cup\{\infty\}$, $G_2:\RR^n\to\RR\cup\{\infty\}$ and $A\in\RR^{n\times m}$.
The ADMM splits the problem into the following smaller subproblems which have to be solved alternatingly.
\\[1ex]
{\bf General ADMM Algorithm:}
Initialization: $y^{(0)}, b^{(0)}\in\RR^m$ and $\gamma>0$.\\
For $i=0,\ldots$ repeat until a convergence criterion is reached:
\begin{align*}
  x^{(i+1)}
&=
  \argmin_{x\in\RR^m}
  \left\{
    G_1(x) + \frac{1}{2\gamma} \|b^{(i)} + Ax - y^{(i)}\|_2^2
  \right\},
  \\
  y^{(i+1)}
&=
  \argmin_{y\in\RR^n}
  \left\{
    G_2(y) + \frac{1}{2\gamma} \|b^{(i)} + Ax^{(i+1)} - y\|_2^2
  \right\},
  \\
  b^{(i+1)}
&=
  b^{(i)}
  +
  Ax^{(i+1)} - y^{(i+1)}.
\end{align*}

To apply this general ADMM algorithm we rewrite the problem (\ref{eq:problemSegmentation}) as
\begin{equation}\label{eq:problemSegmentationADMM}
  \argmin_{u,w\in\RR^{qN^2},v\in\RR^{\eta q N^2}}
  \{\langle
	u,s
  \rangle
  +
  \lambda \| v \|_1 + \iota_C(w) \}
  \text{ subject to }
  (I_q\otimes \mathcal{S})u = v,\ I_{qN^2}u = w
\end{equation}
where $\eta$ is the number of scales and shears in the shearlet transform.
Using
$G_1(u):= \langle u,s \rangle$,
$G_2(v,w) := \lambda \|v\|_1 + \iota_C(w)$ and
$A := \begin{pmatrix}I_q\otimes \mathcal{S} \\ I_{qN^2}\end{pmatrix}$
we obtain the following algorithm:
\\[1ex]
{\bf ADMM Algorithm for (\ref{eq:problemSegmentationADMM}):}
Initialization: $v^{(0)}\in\RR^{\eta q N^2}, w^{(0)}\in\RR^{qN^2}, b^{(0)}\in\RR^{(\eta+1)qN^2}$ and $\gamma>0$.\\
For $i=0,\ldots$ repeat until a convergence criterion is reached:
\begin{align}\label{eq:ADMM_u}
  u^{(i+1)}
&=
  \argmin_{u\in\RR^{qN^2}}
  \left\{
    \langle u,s \rangle + \frac{1}{2\gamma} \|b^{(i)} + Au - \begin{pmatrix}v^{(i)}\\ w^{(i)}\end{pmatrix}\|_2^2
  \right\},
  \\ \label{eq:ADMM_vw}
  \begin{pmatrix}v^{(i+1)}\\ w^{(i+1)}\end{pmatrix}
&=
  \argmin_{v\in\RR^{\eta qN^2},w\in\RR^{qN^2}}
  \left\{
    \lambda \| v \|_1 + \iota_C(w) + \frac{1}{2\gamma} \|b^{(i)} + Au^{(i+1)} - \begin{pmatrix} v \\ w\end{pmatrix}\|_2^2
  \right\},
  \\ \label{eq:ADMM_b}
  b^{(i+1)}
&=
  b^{(i)}
  +
  Au^{(i+1)} -\begin{pmatrix}v^{(i+1)}\\ w^{(i+1)}\end{pmatrix}.
\end{align}
The last step (\ref{eq:ADMM_b}) is a simple update of $b$ and can be split into two parts: one for ``$v$'' and one for ``$w$''. We obtain
\begin{align*}
  b_v^{(i+1)}
&=
  b_v^{(i)}
  +
  (I_q\otimes \mathcal{S})u^{(i+1)} - v^{(i+1)},
  \\
  b_w^{(i+1)}
&=
  b_w^{(i)} + I_{qN^2}u^{(i+1)} - w^{(i+1)}.
\end{align*}
We have to comment the first two steps. Due to the differentiability of (\ref{eq:ADMM_u})
its solution follows by setting the gradient of the functional on the right-hand side to zero. Thus, the minimizer is given by the solution of
\begin{align*}
  0
&=
  s + \frac{1}{\gamma} A^\tT\left(Au + b^{(i)} - \begin{pmatrix}v^{(i)}\\ w^{(i)}\end{pmatrix}\right)
  \\
  A^\tT Au
&=
  A^\tT\left(\begin{pmatrix}v^{(i)}\\ w^{(i)}\end{pmatrix} - \begin{pmatrix}b_v^{(i)}\\ b_w^{(i)}\end{pmatrix} \right) - \gamma s.
\end{align*}
By \eqref{eq:important} we obtain
\begin{equation*}
  A^\tT A
=
  (I_q\otimes \mathcal{S}^\tT, I_{qN^2})\begin{pmatrix}I_q\otimes \mathcal{S} \\ I_{qN^2}\end{pmatrix}
=
  I_q \otimes \underbrace{\mathcal{S}^\tT\mathcal{S}}_{I_N^2} + I_{qN^2}
=
  2 I_{qN^2},
\end{equation*}
so that no linear system has to be solved and we get simply
\begin{align*}
  u
&=
  \frac{1}{2}A^\tT\left(\begin{pmatrix}v^{(i)}\\ w^{(i)}\end{pmatrix} - \begin{pmatrix}b_v^{(i)}\\ b_w^{(i)}\end{pmatrix} \right) - \frac{1}{2}\gamma s
  \\
&=
  \frac{1}{2}
  \left[
	(I_q \otimes \mathcal{S})^\tT(v^{(i)}-b_v^{(i)})
  +
	I_{qN^2}(w^{(i)}-b_w^{(i)}) - \gamma s
  \right]
\end{align*}
which means that we have just to apply an inverse shearlet transform.

The second step (\ref{eq:ADMM_vw}) can be minimized for $v$ and $w$ separately. Thus, we have
\begin{align*}
  v^{(i+1)}
&=
  \argmin_{v\in\RR^{\eta qN^2}}
  \Bigl\{
    \| \Lambda v \|_1 + \frac{1}{2\gamma} \|v - \underbrace{(b_v^{(i)} + (I_q\otimes \mathcal{S})u^{(i+1)})}_{=:g_v}\|_2^2
  \Bigr\}
  \\
&=
  \argmin_{v\in\RR^{\eta qN^2}}
  \left\{
    \gamma \| \Lambda v \|_1 + \frac{1}{2} \|v - g_v\|_2^2
  \right\}
  \\
&=
  S_{\gamma\Lambda}(g_v)\quad\text{(componentwise soft-shrinkage)}
\end{align*}
and
\begin{align*}
  w^{(i+1)}
&=
  \argmin_{w\in\RR^{qN^2}}
  \Bigl\{
    \iota_C(w) + \frac{1}{2\gamma} \|w - \underbrace{(b_w^{(i)} + I_{qN^2}u^{(i+1)})}_{=:g_w}\|_2^2
  \Bigr\}
  \\
&=
  P_C(g_w) \quad (\text{``pixelwise'' orthogonal projection onto } C).
\end{align*}

The {\it soft-shrinkage} $S_\lambda$ with a threshold $\lambda$ is defined as
\begin{equation*}
	S_\lambda(x)
:=
  \begin{cases}
	  x-\lambda&\text{for }x>\lambda,\\
	  0		   &\text{for }x\in[-\lambda,\lambda],\\
	  x+\lambda&\text{for }x<-\lambda.
  \end{cases}
\end{equation*}
The set $C$ is a subset of the hyperplane
\begin{equation*}
\{x\in\RR^q : \sum_{k=1}^qx_k = \langle x,\mathbf{1}_q \rangle = 1\} \subset \RR^q
\end{equation*}
The {\it orthogonal projection} of $g\in\RR^q$ onto the hyperplane is given by
\begin{equation*}
  \widetilde{g} = g + \frac{1}{q}\left(1-\sum_{k=1}^q g_k\right) \mathbf{1}_q.
\end{equation*}
This can easily be constructed geometrically by  calculating the intersection point of the straight line through $g$ with slope $\mathbf{1}_C$ and the hyperplane.
It cannot be guaranteed that $\widetilde{g}$ lies in $C$, i.e., some components could be less than zero. The projection onto $C$ can be achieved by setting these components
to zero and project again on a lower-dimensional space and so on.
\\

Finally we remark that primal-dual methods as in \cite{CP11,ZC08} could be applied instead of the ADMM.
However, due to the Parseval frame property which preserves us from
solving a linear system of equations in the first ADMM step, the primal-dual methods are not superior here.

\section{Numerical Examples}\label{sec:num_results}
In this section we demonstrate the performance of our algorithm by numerical examples.
We will compare our method with minimization models using the same
data term, but a different regularizer namely the classical TV regularizer
and an NL regularizer.
While the TV-functional is well-suited for cartoon-like parts,
NL-means are a powerful tool for restoring textured regions in images \cite{BCM05,GO07,GO09}.

There exist different possibilities to apply the TV-like functional in segmentation.
Here we focus on the method proposed in \cite{LBS09} which can be written in matrix-vector notation as
\begin{equation*} \label{eq:TVreg}
{\rm TV} (u) := \| \, | (I_q \otimes \nabla) u | \, \|_1
= \sum_{j=1}^N  | (I_q \otimes \nabla) u |(j)
= \sum_{j=1}^N \left( |\nabla u_1(j) |^2 + \cdots + |\nabla u_q (j) |^2 \right)^{\frac12},
\end{equation*}
see \cite{SS11}. As mentioned one could use alternatively
$\sum_{j=1}^N ( |\nabla u_1(j) | + \cdots + |\nabla u_q (j) | )$,
where we have not seen a visually different result in our applications.

Concerning the NL regularizer, we have used the approach in \cite{GO07a,ST10} with respect to the
$q$ labels. This means that we replace the discrete gradient matrix $\nabla$ by a matrix
$D$ which is constructed as follows:
\\
Initially, we start with a zero weight matrix $w$.
For every image pixel $i = (i_1,i_2)$, we compute for all $j =(j_1,j_2)$
within a search window of size $\varpi\times\varpi$ around
$i$ the distances
\begin{equation*} 
d_a (i,j) :=
\sum_{t_1 = -\lceil \frac{p-1}{2} \rceil}^{\lceil  \frac{p-1}{2} \rceil}
\sum_{t_2 = -\lceil \frac{p-1}{2} \rceil}^{\lceil  \frac{p-1}{2} \rceil}
G_a(t_1,t_2) \big( f(i_1+t_1,i_2+t_2) - f(j_1+t_1,j_2+t_2) \big)^2,
\end{equation*}
where $G_a$ denotes the discretized, normalized Gaussian with standard deviation $a$.
We refer to $p$ as the patch size.
Then, for given $\widetilde{m}$ we select $k \leq \widetilde{m}$ so-called 'neighbors' $j\neq i$ of $i$ for which
$d_a(i,j)$ takes the smallest values and set $w(i,j) = w(j,i) := 1$.
By setting $w(i,j)= w(j,i)$ it happens that several weights $w(j,\cdot)$ are already non-zero before we reach pixel $j$. To avoid that the number of non-zero weights becomes too large, we choose only $k = \min\{\widetilde{m}, 2\widetilde{m} -l \}$ neighbors for $l$ being the number of non-zero weights $w(j,\cdot)$ before the selection.
Now we construct the matrix $D \in \mathbb R^{mN,N}$ with $m = 2\widetilde{m}$ so that
$D$ consists of $m$ blocks of size $N \times N$, each having $-1$ as diagonal elements plus one additional
nonzero value $1$ in each row whose position is determined by the nonzero weights $w(i,j)$
and maybe some zero rows.
The following parameters were used throughout the experiments: $p=5, \varpi=15, a=2, \widetilde{m}=5$.\\

First we present an artificial example. In Fig. \ref{fig:gitter}
an original $0-1$ grid image and its noisy version with white Gaussian noise of standard deviation $0.2$ are shown.
The noisy image was segmented by minimizing the proposed convex functional with TV regularization (bottom left)
and shearlet regularization (bottom right). As codebook we have used the exact values
0 and 1.
While the first method fails to reconstruct the grid, our
shearlet segmentation method works perfect.
An exact reconstruction can be also obtained by applying the NL-means regularizer.
However in its isotropic version the later method requires more time than the shearlet approach.
%
\begin{figure}[htbp]
	\centering
	\subfigure[Original image]{\label{fig:gitterOrig}
\includegraphics[width=0.4\textwidth]{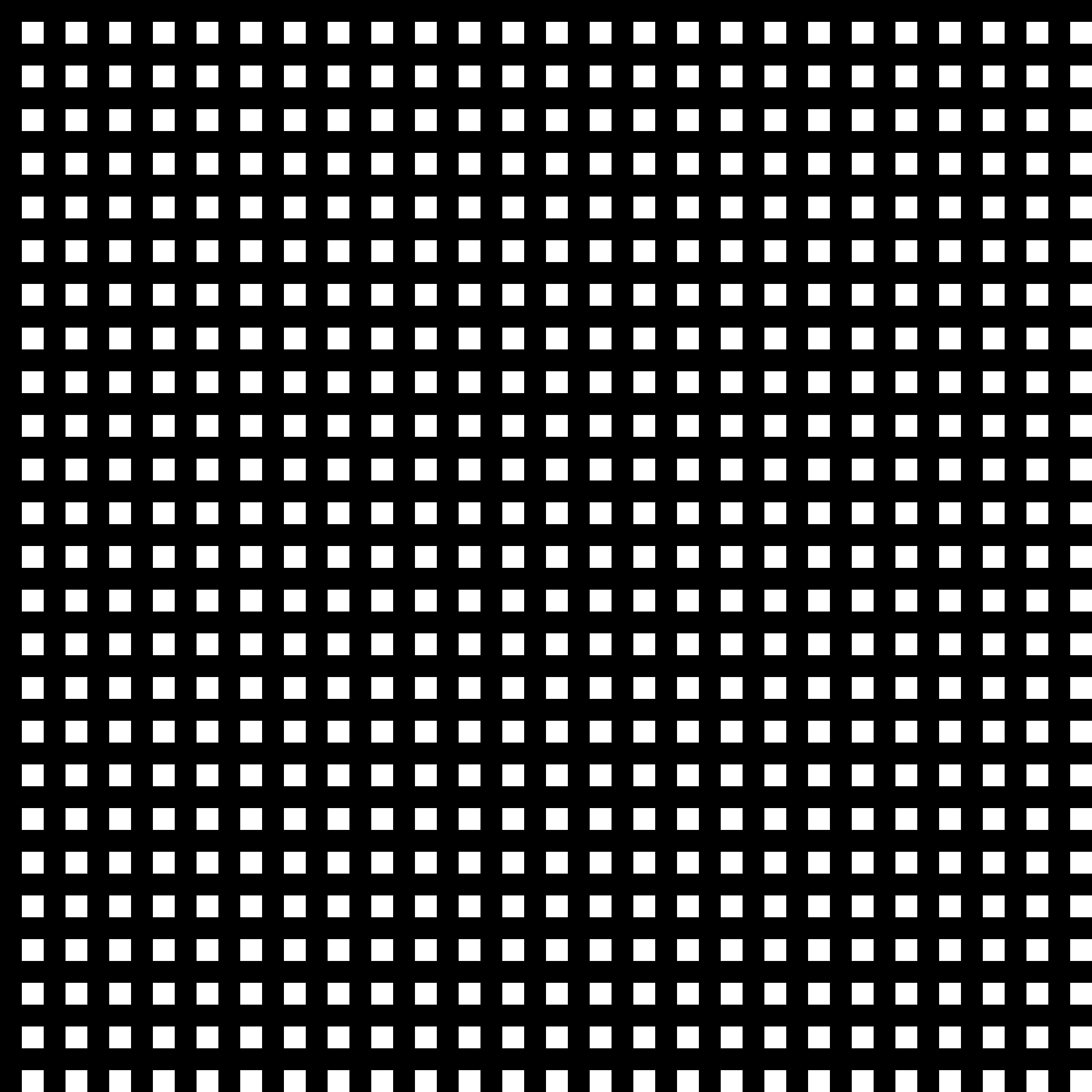}}
	\hspace{1cm}
	\subfigure[Noisy image]{\label{fig:gitterNoisy}
\includegraphics[width=0.4\textwidth]{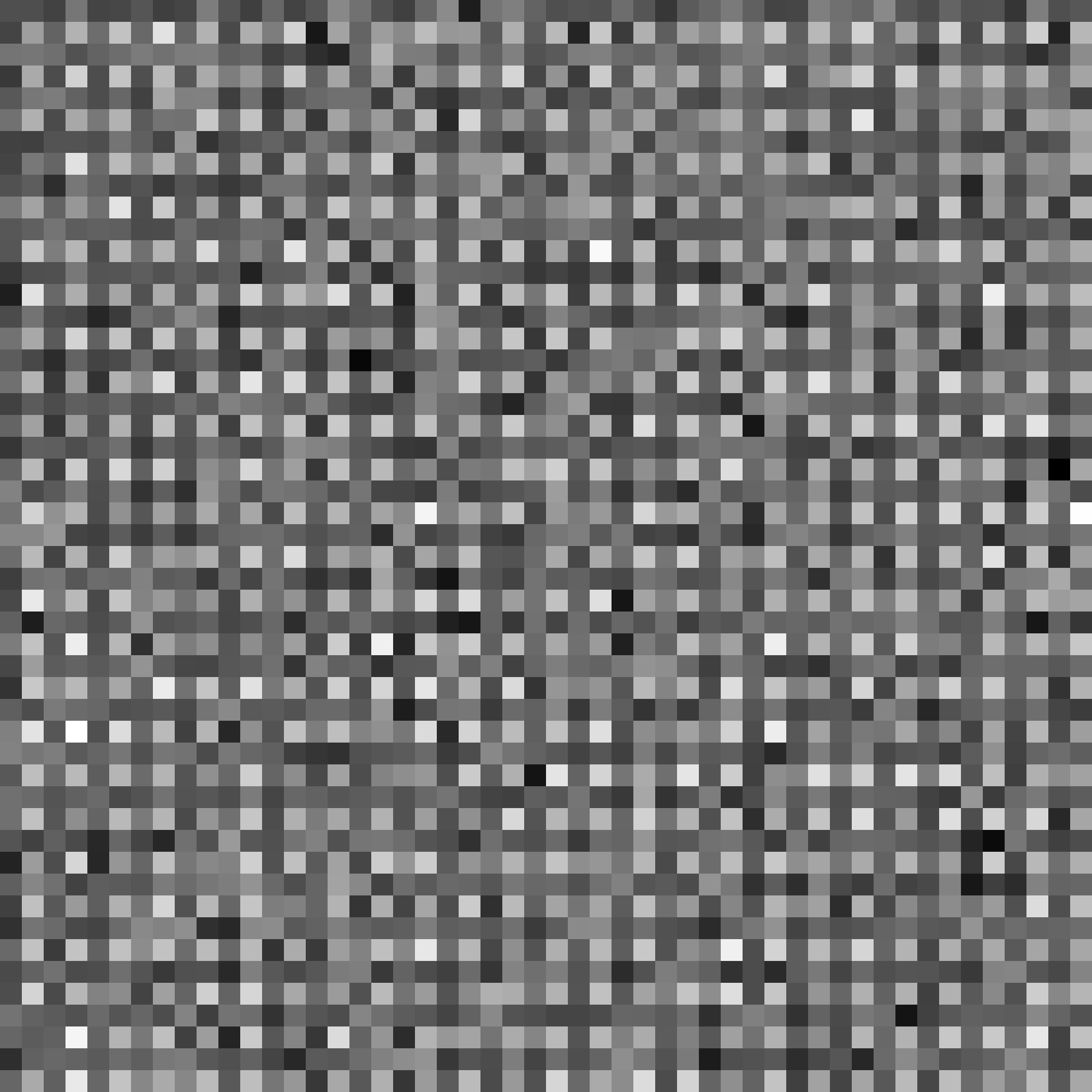}}
\\
	\subfigure[Segmented image with TV regularizer, \newline$\lambda = \frac{1}{60}$, $\gamma = \frac{1}{4}$, 300 iterations]{\label{fig:gitterTV}
\includegraphics[width=0.4\textwidth]{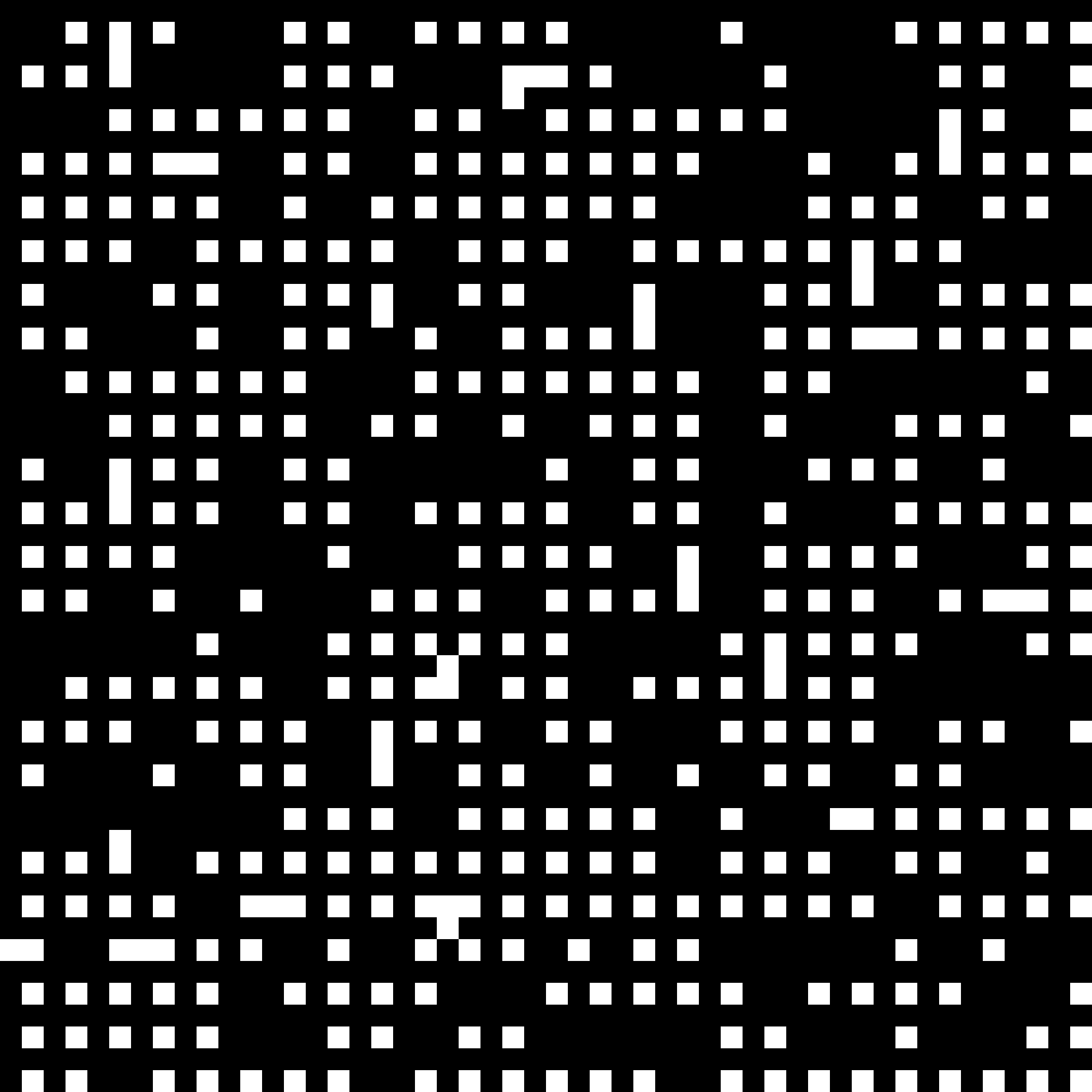}}
	\hspace{1cm}
	\subfigure[Segmented image with shearlet regularizer, \newline $\lambda = \frac{1}{512}$, $\gamma = \frac{1}{20}$, 10 iterations ]{\label{fig:gitterShearlet}
\includegraphics[width=0.4\textwidth]{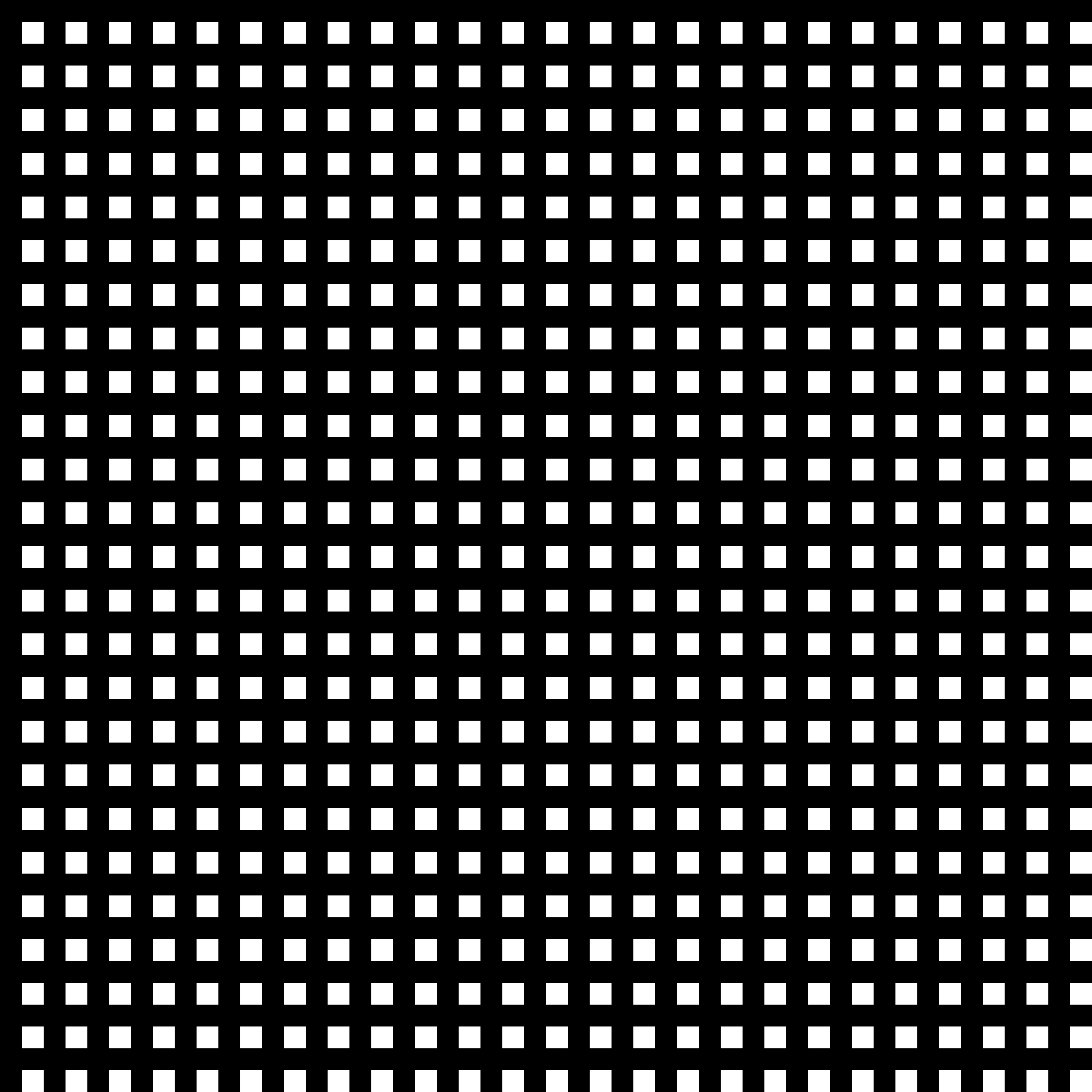}}
	\caption{Segmentation of a noisy grid image with TV regularization and shearlet regularization.}
	\label{fig:gitter}
\end{figure}
\\
Next we want to segment the cartoon type image in Fig. \ref{fig:cartoon}.
The top left image (Fig.  \ref{fig:cartoon_original}) shows the original cartoon image with four constant regions.
The image was taken from \cite{CCZ12}. In the top right image white Gaussian noise of standard deviation $0.1$ was added (Fig. \ref{fig:cartoon_noise_01}).
The bottom row shows the difference between the segmented image and the original one.
For the left image (Fig. \ref{fig:cartoon_noise_01_TV}) the TV regularizer was used and for the right image (Fig. \ref{fig:cartoon_noise_01_shearlet}) the shearlet regularizer.
In both images a white pixel represents the correctly labeled pixels and the black pixels indicate false labeled pixels.
Both methods have problems with correctly labeling the borders of the circles and
the TV method also fails on assigning the correct labels for the diagonal border of the
triangle whereas the shearlet regularizer labels all these points in the right way.
\begin{figure}[htbp]
  \centering
  \subfigure[Original cartoon image]
  {\label{fig:cartoon_original}
    \includegraphics[width=0.38\textwidth]{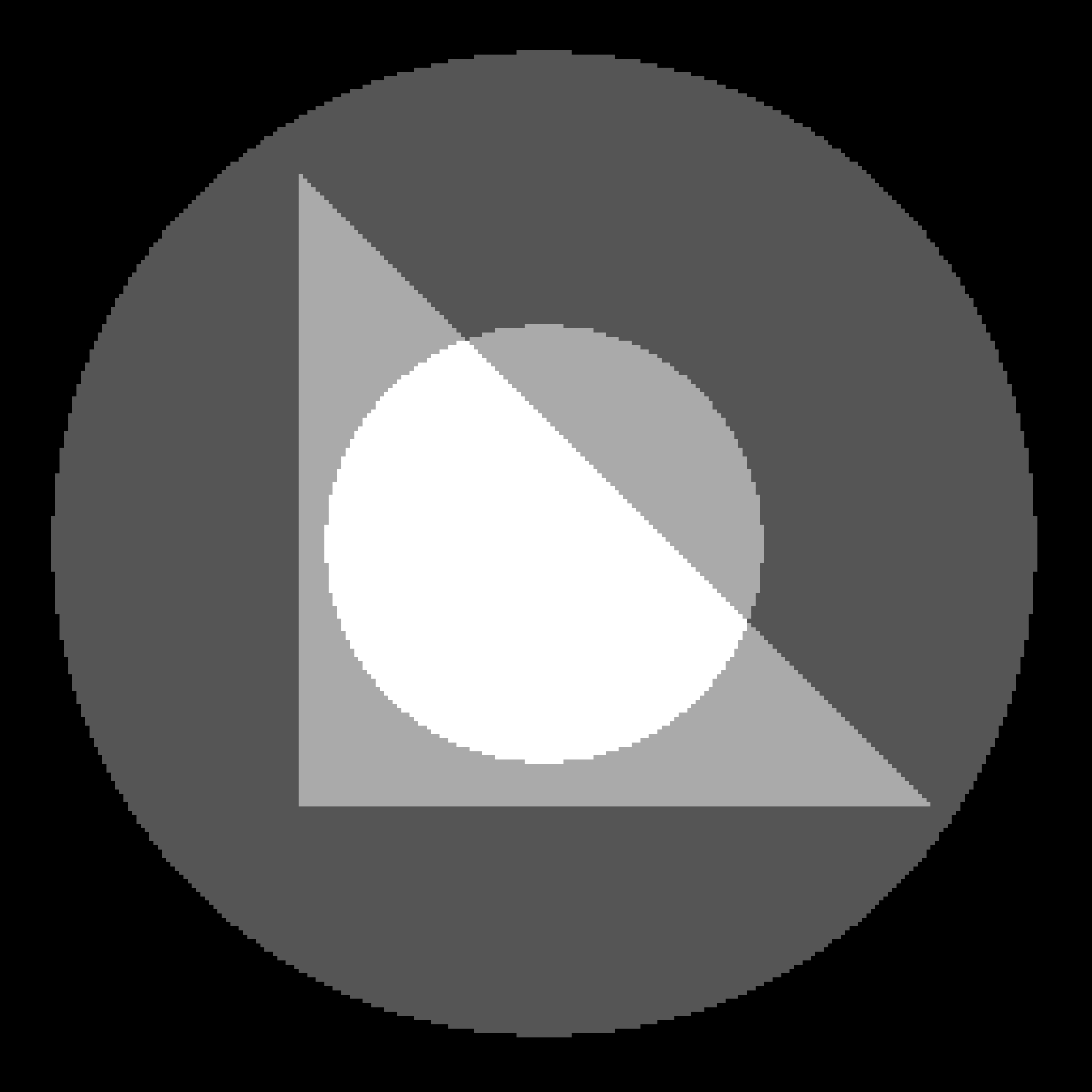}
  }
  \hspace{1cm}
  \subfigure[Cartoon image, noise $\sigma = 10\%$]
  {\label{fig:cartoon_noise_01}
    \includegraphics[width=0.38\textwidth]{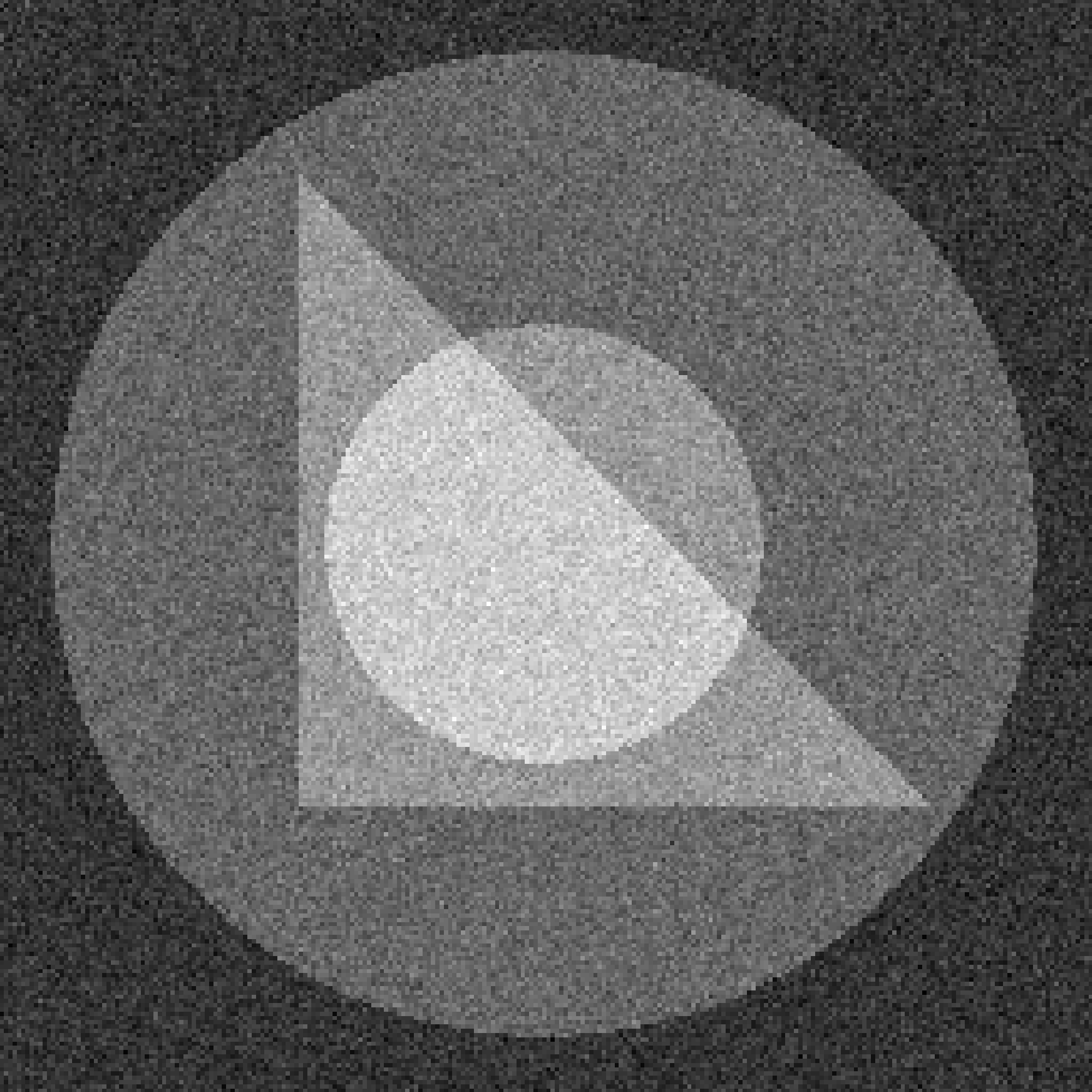}
  }
  \\
  \subfigure[Difference between original and segmented image using TV regularized segmentation, $\gamma = 1$, $\lambda = \frac{1}{6}$, 100 iterations]
  {\label{fig:cartoon_noise_01_TV}
    \includegraphics[width=0.38\textwidth]{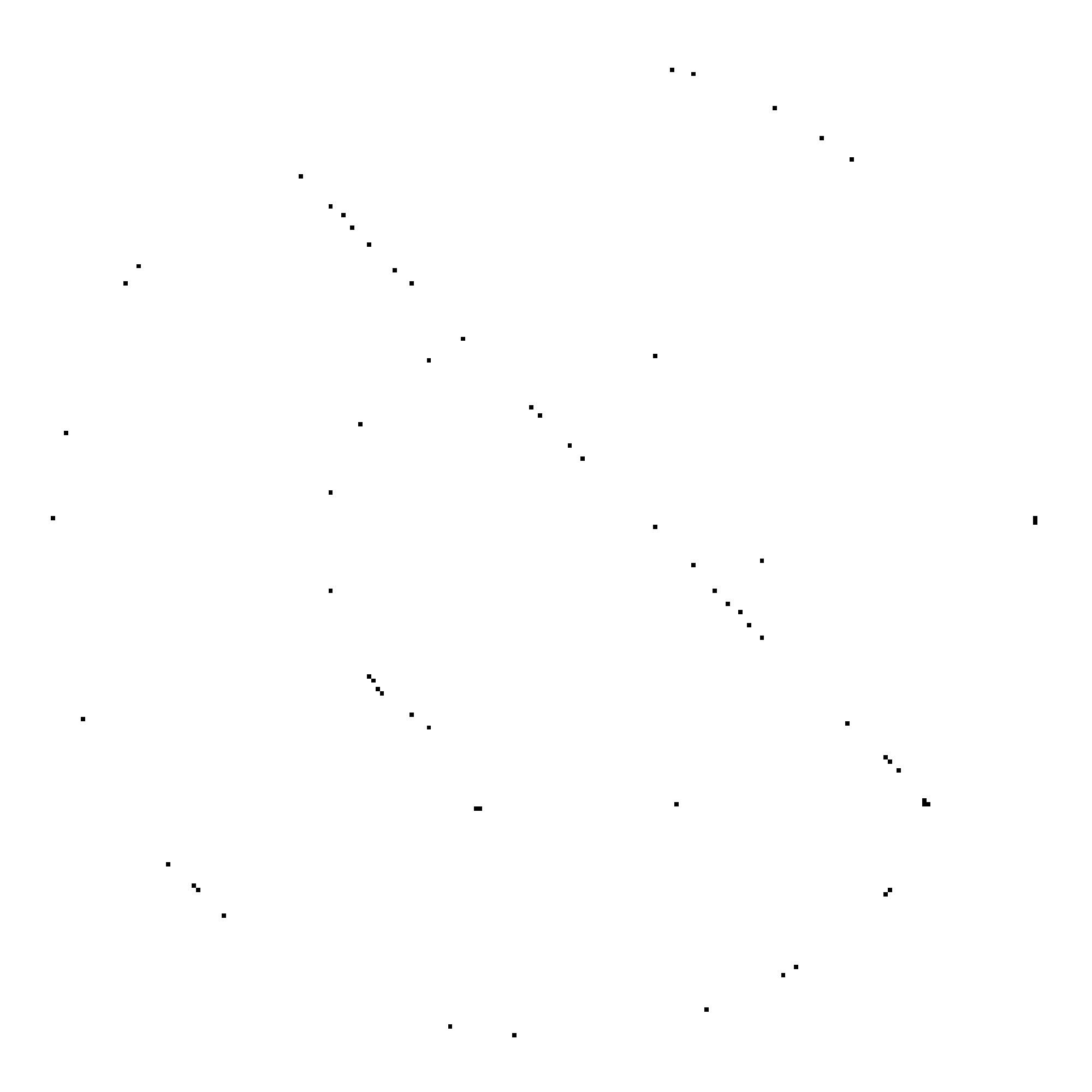}
  }
  \hspace{1cm}
  \subfigure[Difference between original and segmented image using shearlet regularized segmentation, $\gamma = 1$, $\lambda = \frac{1}{20}(0,0.1,0.2,2.20)$, 50 iterations]
  {\label{fig:cartoon_noise_01_shearlet}
    \includegraphics[width=0.38\textwidth]{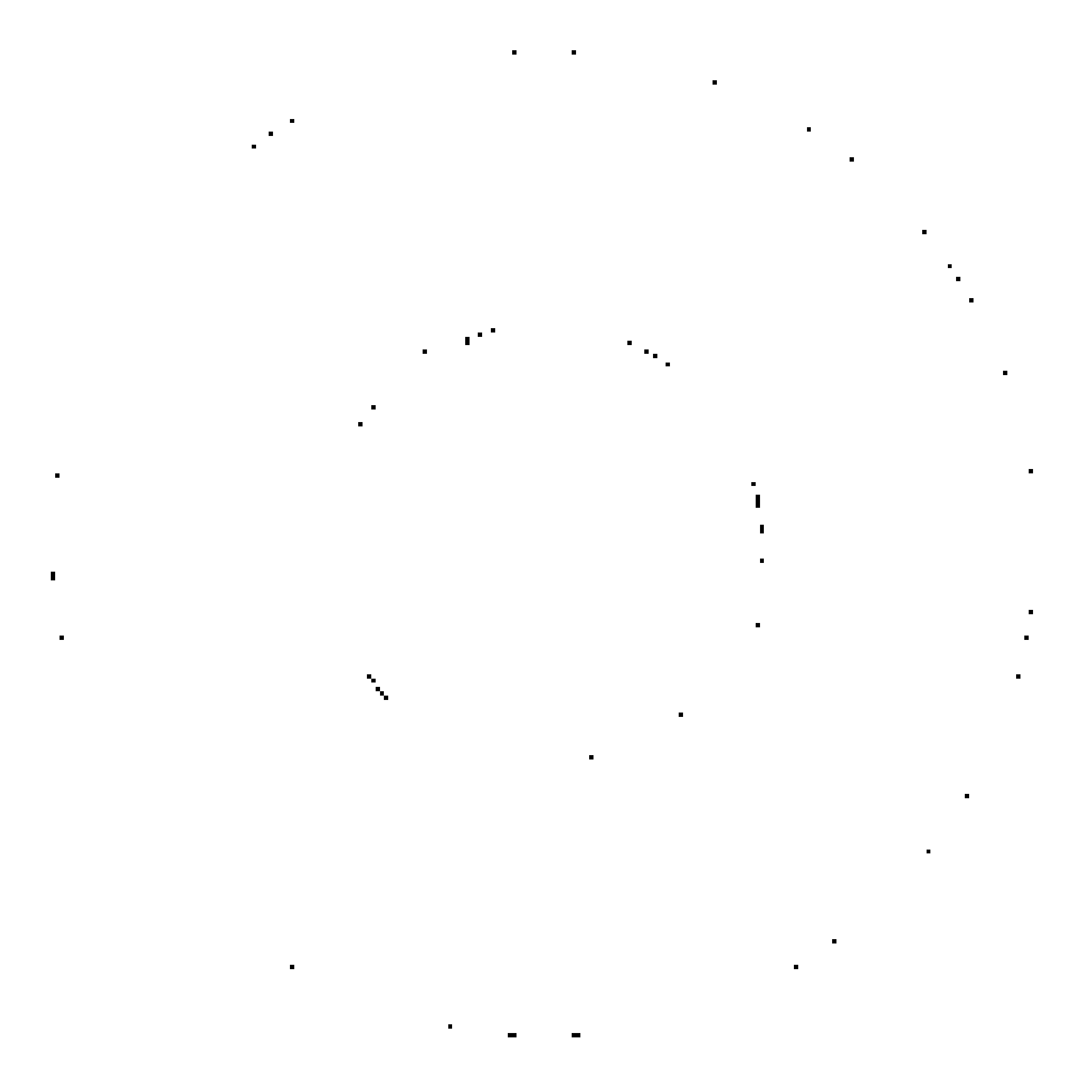}
  }
  \caption{Four-class of a cartoon image}
  \label{fig:cartoon}
\end{figure}

To illustrate the good performance of the shearlet regularizer for curved structures compared to the TV regularizer
we provide another example using the april calender sheet of www.mathe-kalender.de, i.e.,
 a snippet from the right bottom corner (see Fig. \ref{fig:stripes_originial}).
For the segmentation we added again white Gaussian noise of standard deviation $\sigma = 0.1$ in Fig. \ref{fig:stripes_noise}.
The six different reference colors for the codebook were chose manually.
In Fig. \ref{fig:stripes_noise_TV} we show the result using the TV regularizer and in Fig. \ref{fig:stripes_noise_shearlet} we used the shearlet regularizer.
Our shearlet method preserves the stripes better than the first method.
Adding more noise the results using TV become worse whereas the results using shearlets remain similar.

\begin{figure}[htbp]
  \centering
  \subfigure[Original image with colored stripes]
  {\label{fig:stripes_originial}
    \includegraphics[width=0.38\textwidth]{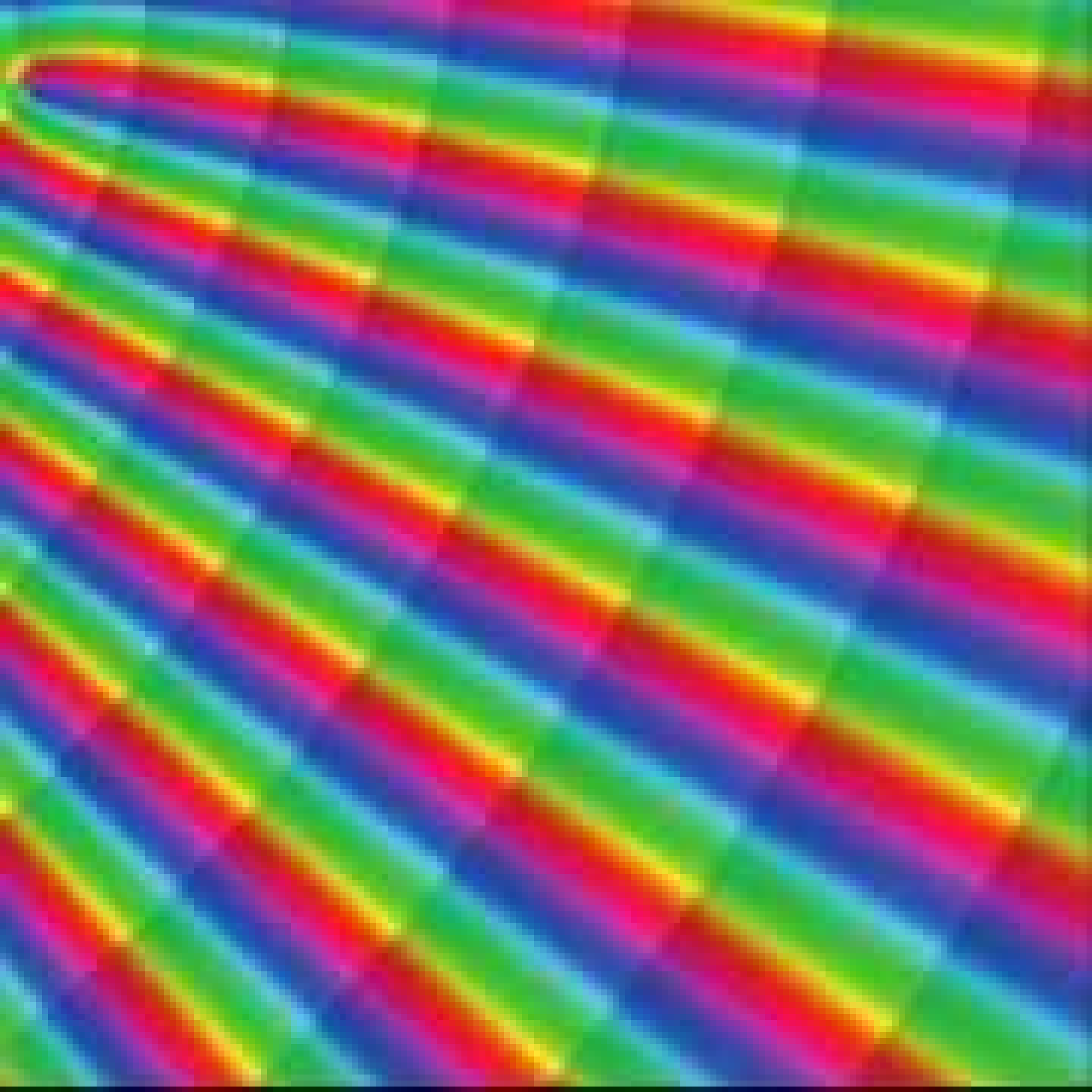}
  }
  \hspace{1cm}
  \subfigure[Image with colored stripes, noise $\sigma = 10\%$]
  {\label{fig:stripes_noise}
    \includegraphics[width=0.38\textwidth]{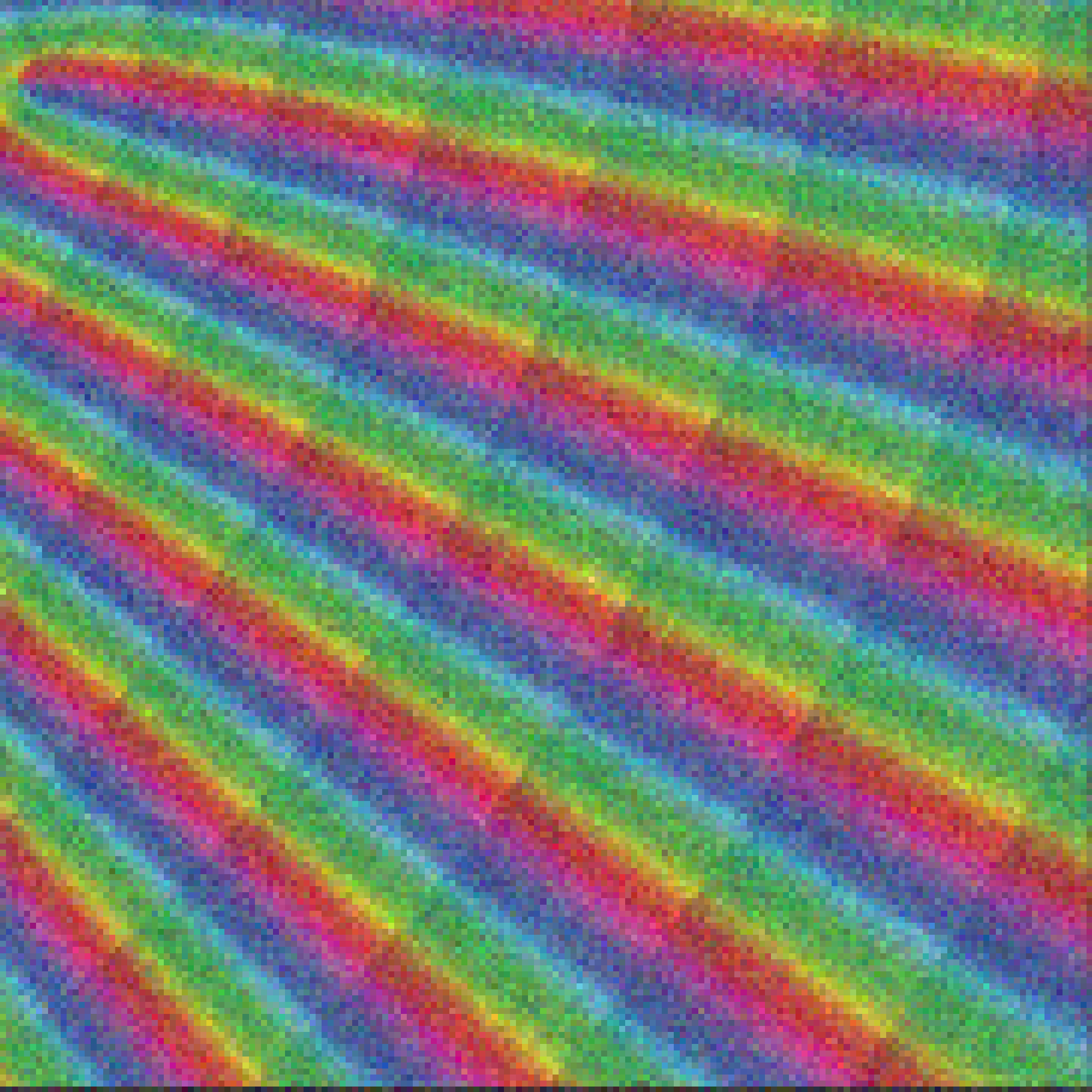}
  }
  \\
  \subfigure[Segmented image using TV regularizer,  $\gamma = 1$, $\lambda = \frac{1}{12}$, 100 iterations]
  {\label{fig:stripes_noise_TV}
    \includegraphics[width=0.38\textwidth]{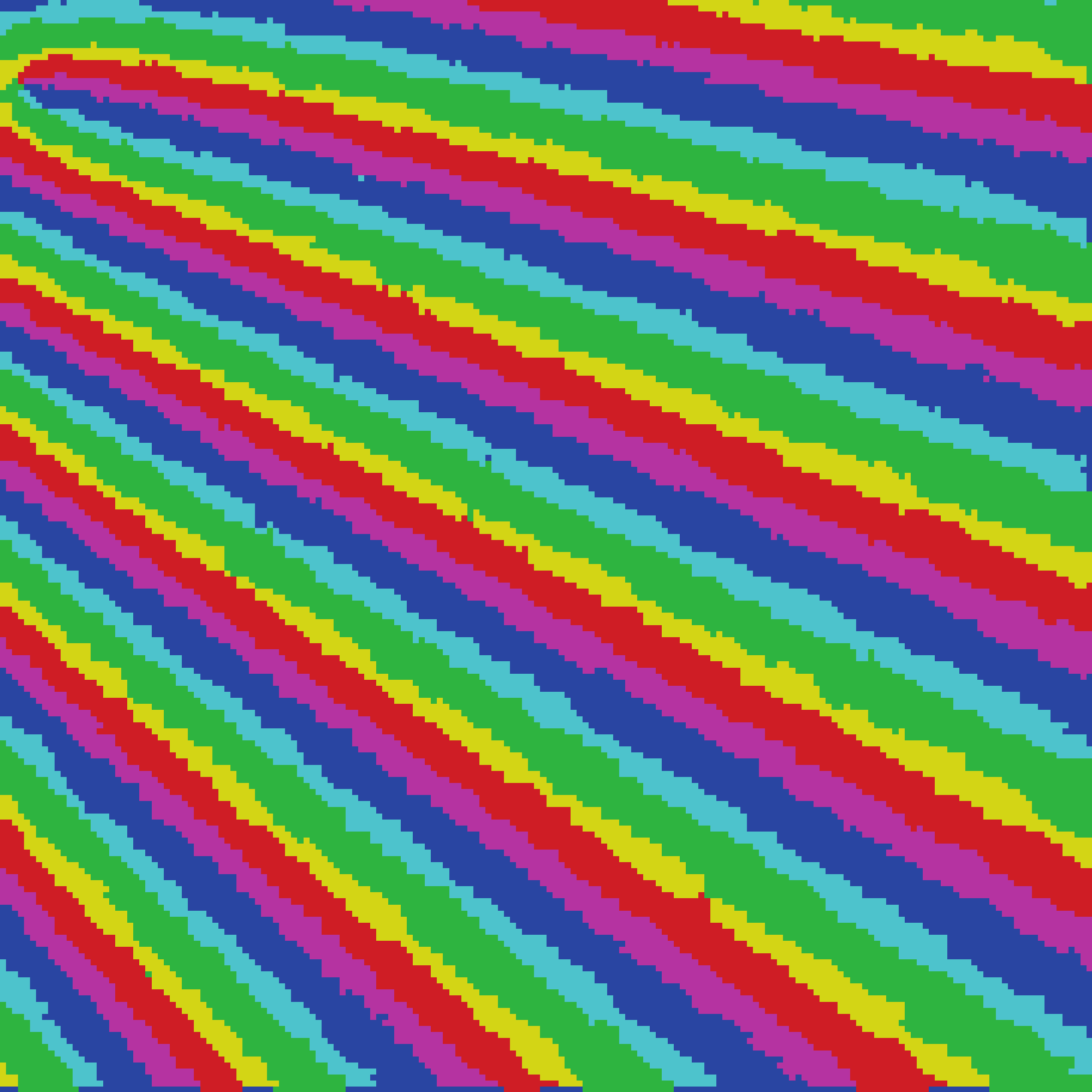}
  }
  \hspace{1cm}
  \subfigure[Segmented image using shearlet regularizer, $\gamma = 1$, $\lambda = \frac{1}{20}(0,0.1,0.2,0.4)$, 50 iterations]
  {\label{fig:stripes_noise_shearlet}
    \includegraphics[width=0.38\textwidth]{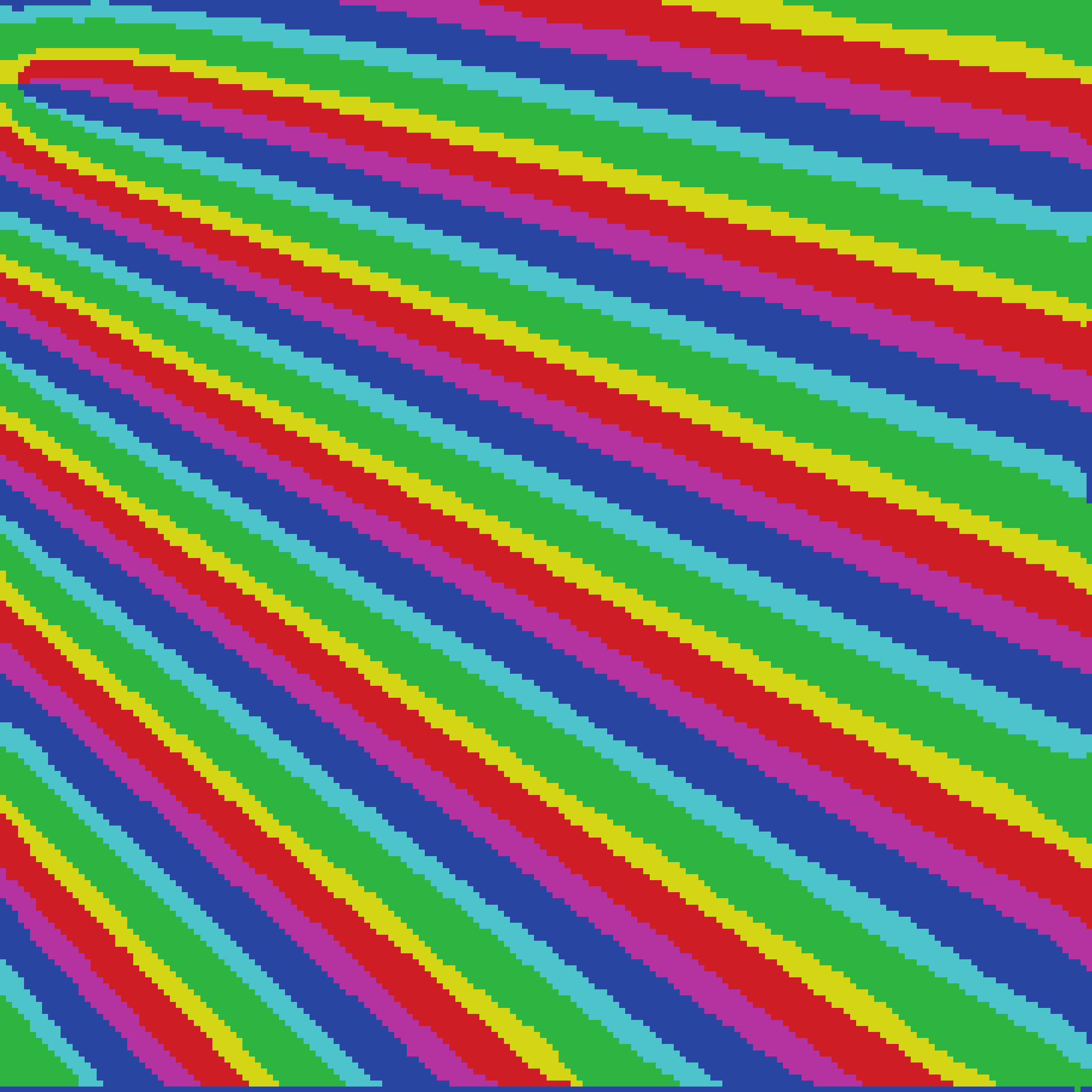}
  }
  \caption{Segmentation of stripes}
  \label{fig:stripes}
\end{figure}

Last we want to segment the real image in Fig. \ref{fig:fish}, where \ref{fig:fishOrig} is the original RGB image
of a clown-surgeon fish
and in \ref{fig:fishNoisy} white Gaussian noise of standard deviation $0.2$ was added.
As codebook we use the matrix
\begin{equation*}
  c
=
  \begin{pmatrix}
    0.7451 &   0.1843 &   0.3686 &   0.8353 \\
    0.8314 &   0.2784 &   0.5569 &   0.7333 \\
    0.8196 &   0.2275 &   0.6353 &   0.3020
  \end{pmatrix}
\end{equation*}
where each row represents a color channel and each column stands for one label.
The colors were chosen manually according to the different regions of the image.
Compared to the segmentation model above for gray valued images only the computation of $s$ in \eqref{eq:computationOfs}
has to be slightly modified. Assuming that every pixel in the image $f$ consists of a three element vector
$(f_R,f_G,f_B)$ we compute $s$ as the norm of the difference between the vectors $(f_R,f_G,f_B)$ and $(c_R,c_G,c_B)$ - in formulas:
\begin{equation*}
	s[r+(k-1)N^2]
=
  \left\|
  \begin{pmatrix} f_R[r] \\ f_G[r] \\ f_B[r] \end{pmatrix}
  -
  \begin{pmatrix} c_R[k] \\ c_G[k] \\ c_B[k] \end{pmatrix}
  \right\|_p^p
\quad\text{for }
  1\leq r\leq N^2,\
  1\leq k\leq q.
\end{equation*}

\begin{figure}[htbp]
	\centering
	\subfigure[Original image]{\label{fig:fishOrig}
\includegraphics[width=0.45\textwidth]{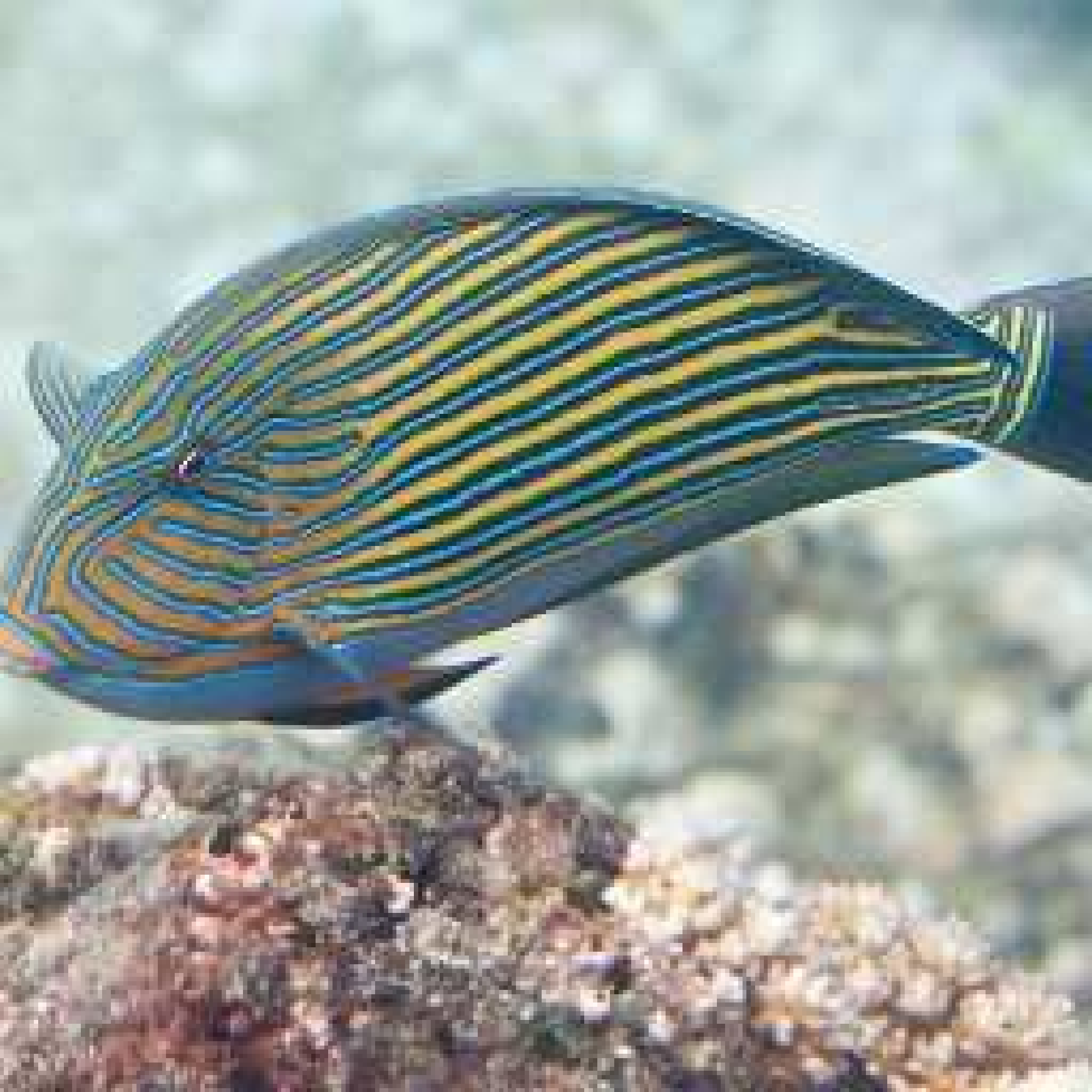}}
	\hspace{1cm}
	\subfigure[Noisy image]{\label{fig:fishNoisy}
\includegraphics[width=0.45\textwidth]{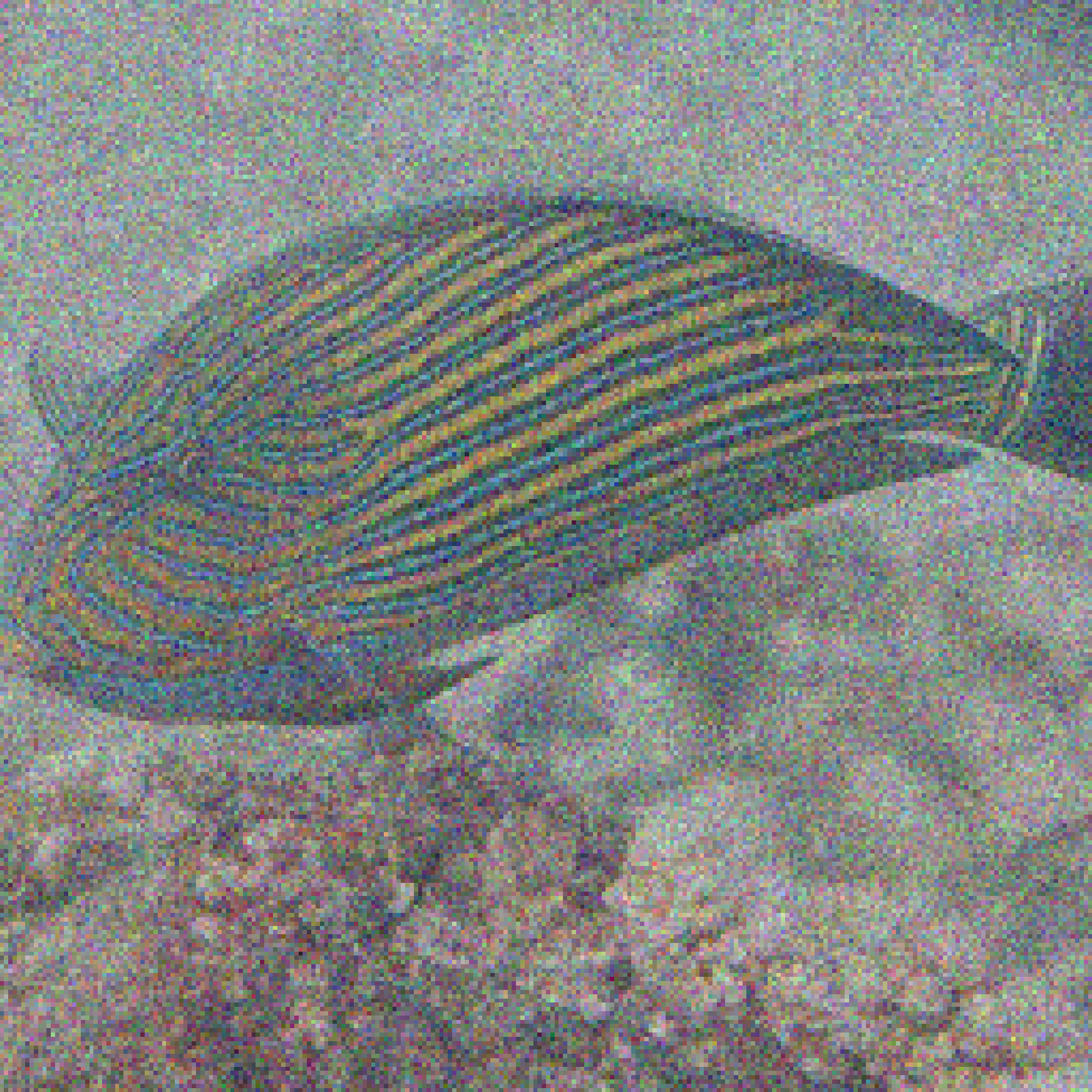}}
	\caption{Original RGB image of a clown-surgeon fish and its noisy version with white Gaussian noise of standard deviation $0.2$.}
	\label{fig:fish}
\end{figure}
In Fig. \ref{fig:fishSegmentation} we compare the segmentation results for both images
when applying different regularizers in the minimization functional.

The first row (Fig. \ref{fig:fishTV} and \ref{fig:fishNoisyTV})
was segmented using the TV regularizer.
The second row (Fig. \ref{fig:fishNL} and \ref{fig:fishNoisyNL})
is obtained with the NL regularizer.
The last row (Fig. \ref{fig:fishShearlet} and \ref{fig:fishNoisyShearlet}) shows our shearlet regularization.
The parameter were chosen such that we get visually the best result.
For the original image the results are rather similar with slight difficulties for the TV-model to
segment the curves in the front of the fish.
For the noisy image the results for the TV-based method get worse.
As expected for different parameters either noise and structure is preserved
or the noise vanishes in the segmentation process but we loose some structure, too.
However, the TV method is the fastest one.
The NL-means regularizer does a good job, but requires more computational time than
our shearlet method and tends to preserve tiny structures.
Our shearlet method segments this kind of curved structure very well.

\begin{figure}[htbp]
    \centering
	\subfigure[TV segmentation\newline $\gamma = 1$, $\lambda = \frac{1}{12}$, 100 iterations]
	{\label{fig:fishTV}
	  \includegraphics[width=0.38\textwidth]{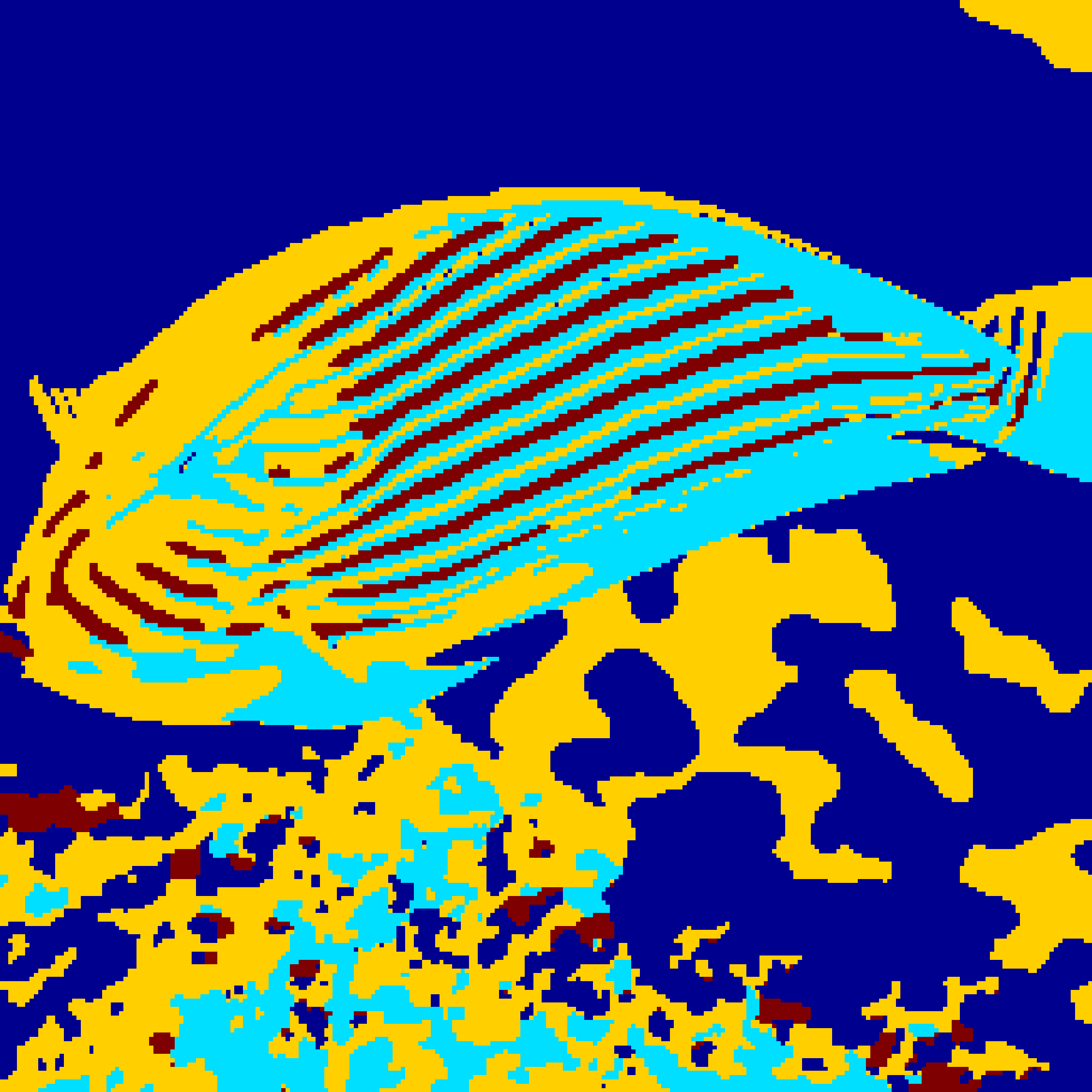}
	}
	\hspace{1cm}
	\subfigure[TV segmentation, noisy image,\newline $\gamma = 1$, $\lambda = \frac{1}{6}$, 100 iterations]
	{\label{fig:fishNoisyTV}
	  \includegraphics[width=0.38\textwidth]{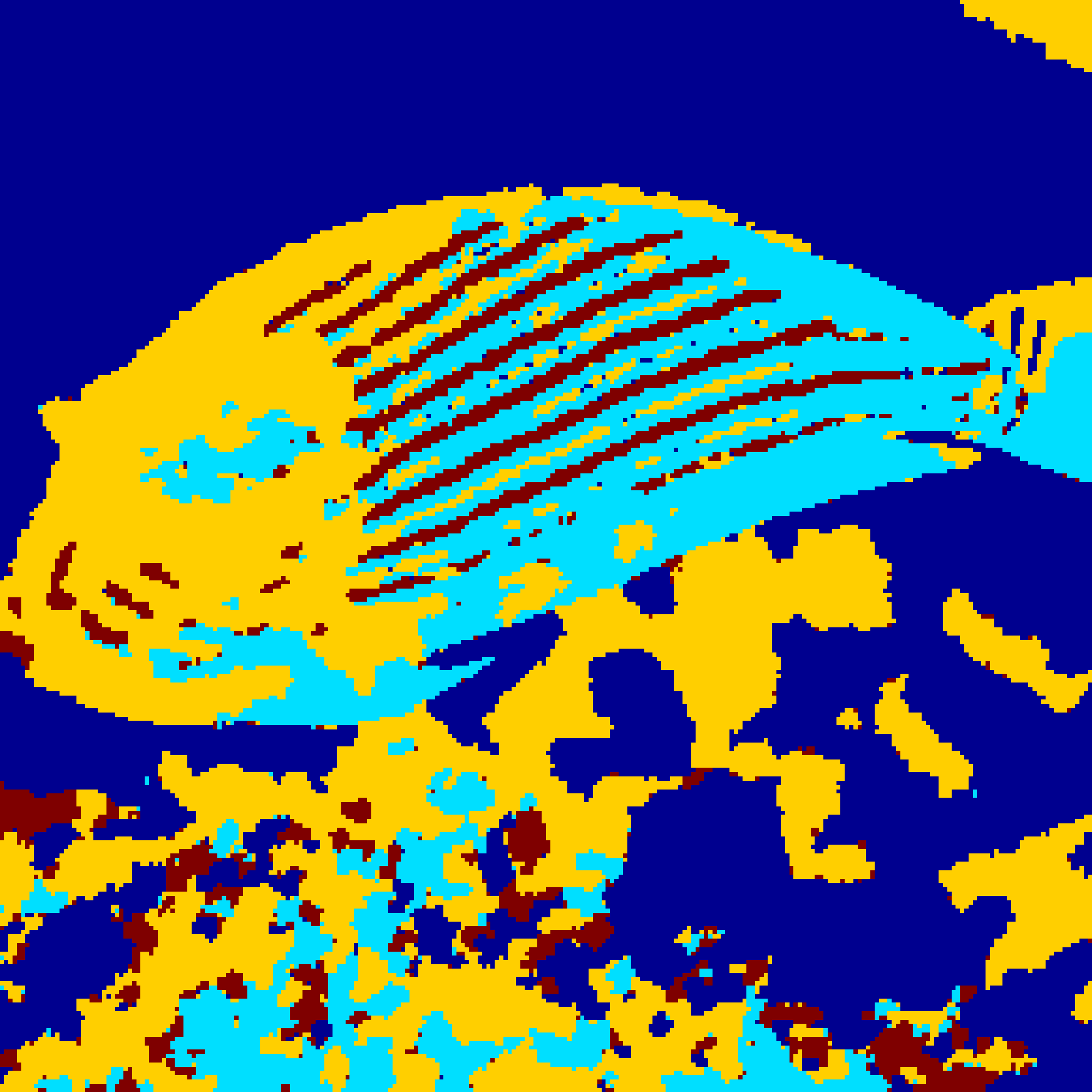}
	}
	\\
    \subfigure[NL-means segmentation\newline $\gamma = \frac{1}{5}$, $\lambda = \frac{1}{32}$, 100 iterations]
	{\label{fig:fishNL}
	  \includegraphics[width=0.38\textwidth]{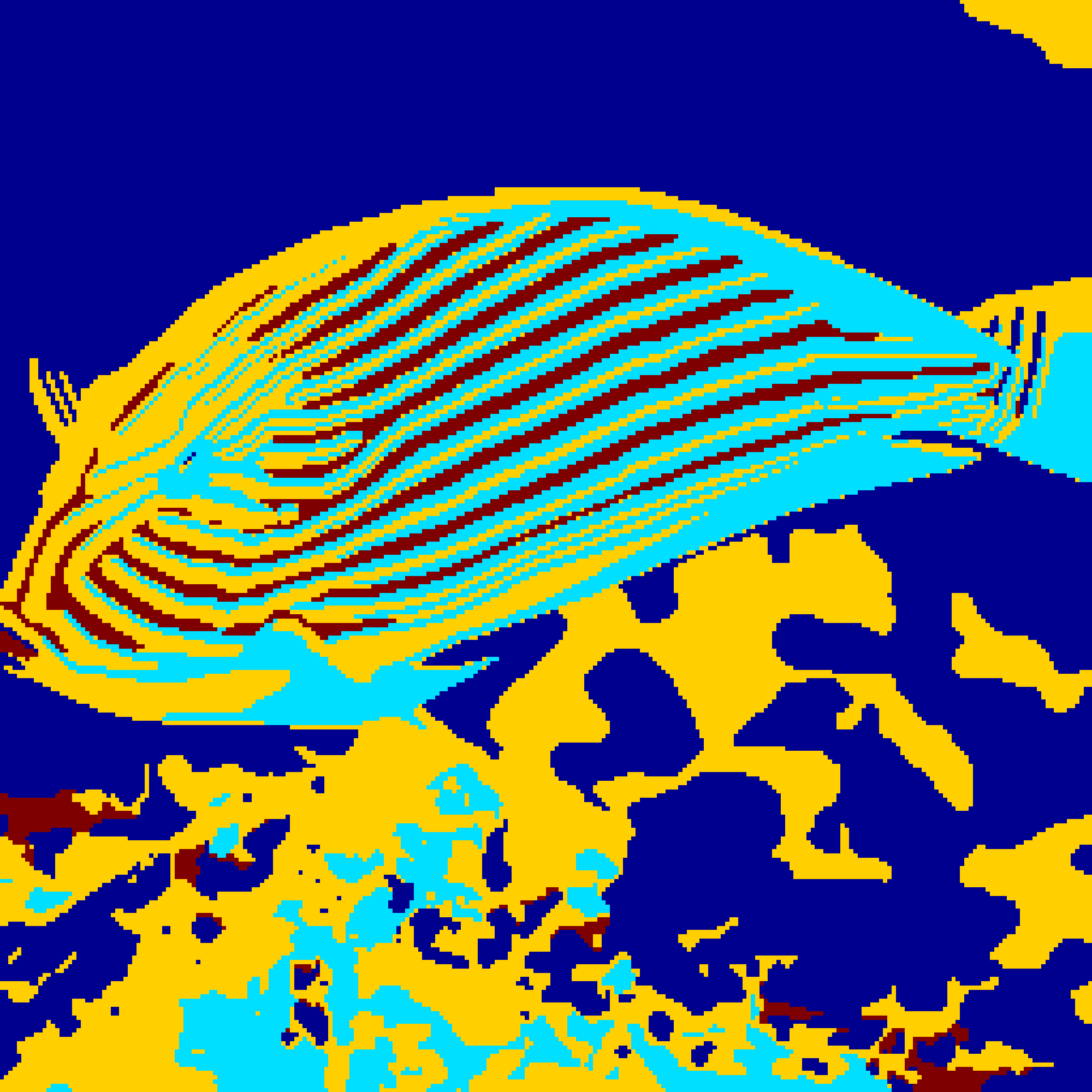}
	}
	\hspace{1cm}
    \subfigure[NL-means segmentation, noisy image\newline $\gamma = \frac{1}{5}$, $\lambda = \frac{1}{128}$, 100 iterations]
	{\label{fig:fishNoisyNL}
	  \includegraphics[width=0.38\textwidth]{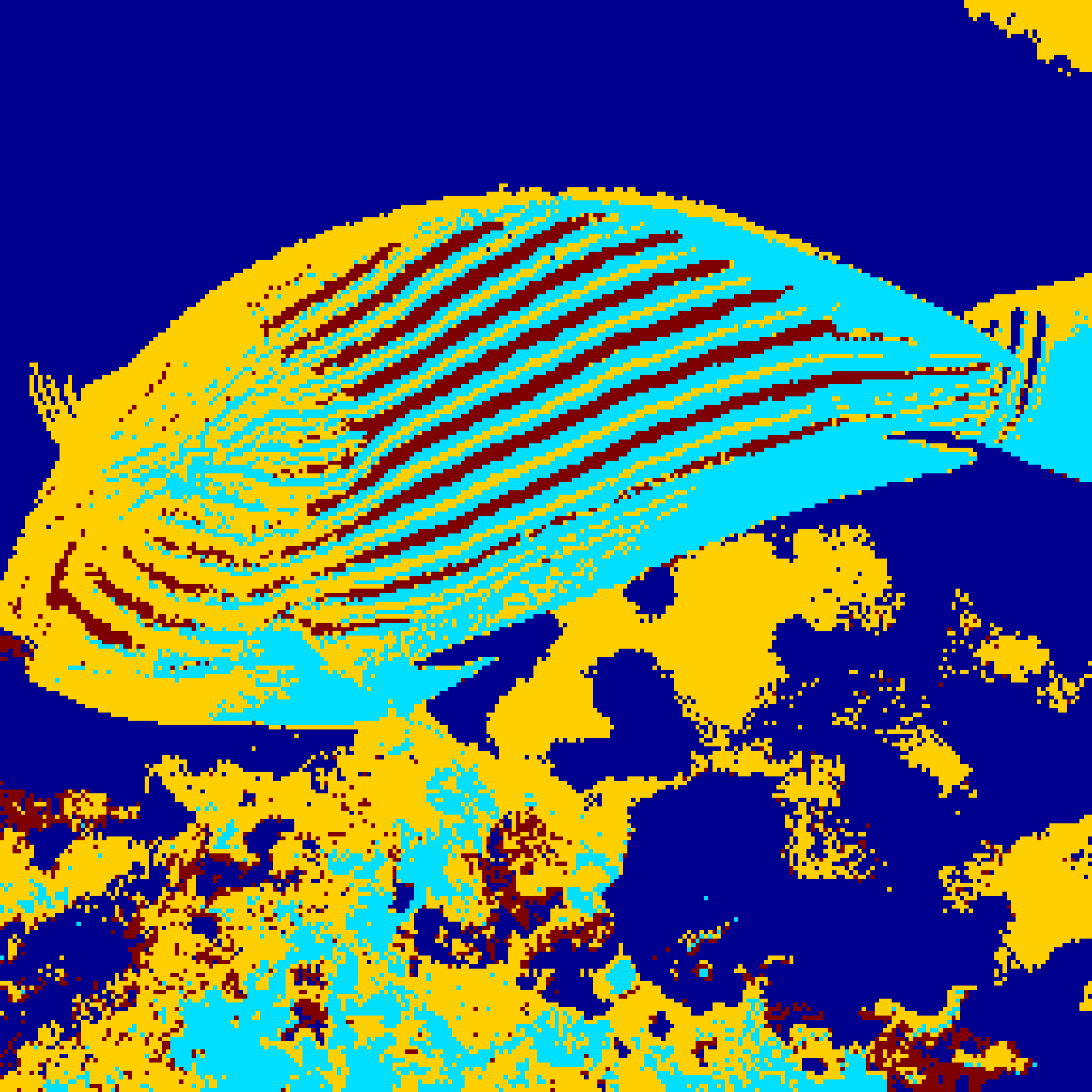}
	}
	\\
	\subfigure[Shearlet segmentation\newline $\gamma = 4$, $\lambda = \frac{1}{250}(0,0.1,0.2,0.4,0.8)$,\newline 30 iterations]
	{\label{fig:fishShearlet}
	  \includegraphics[width=0.38\textwidth]{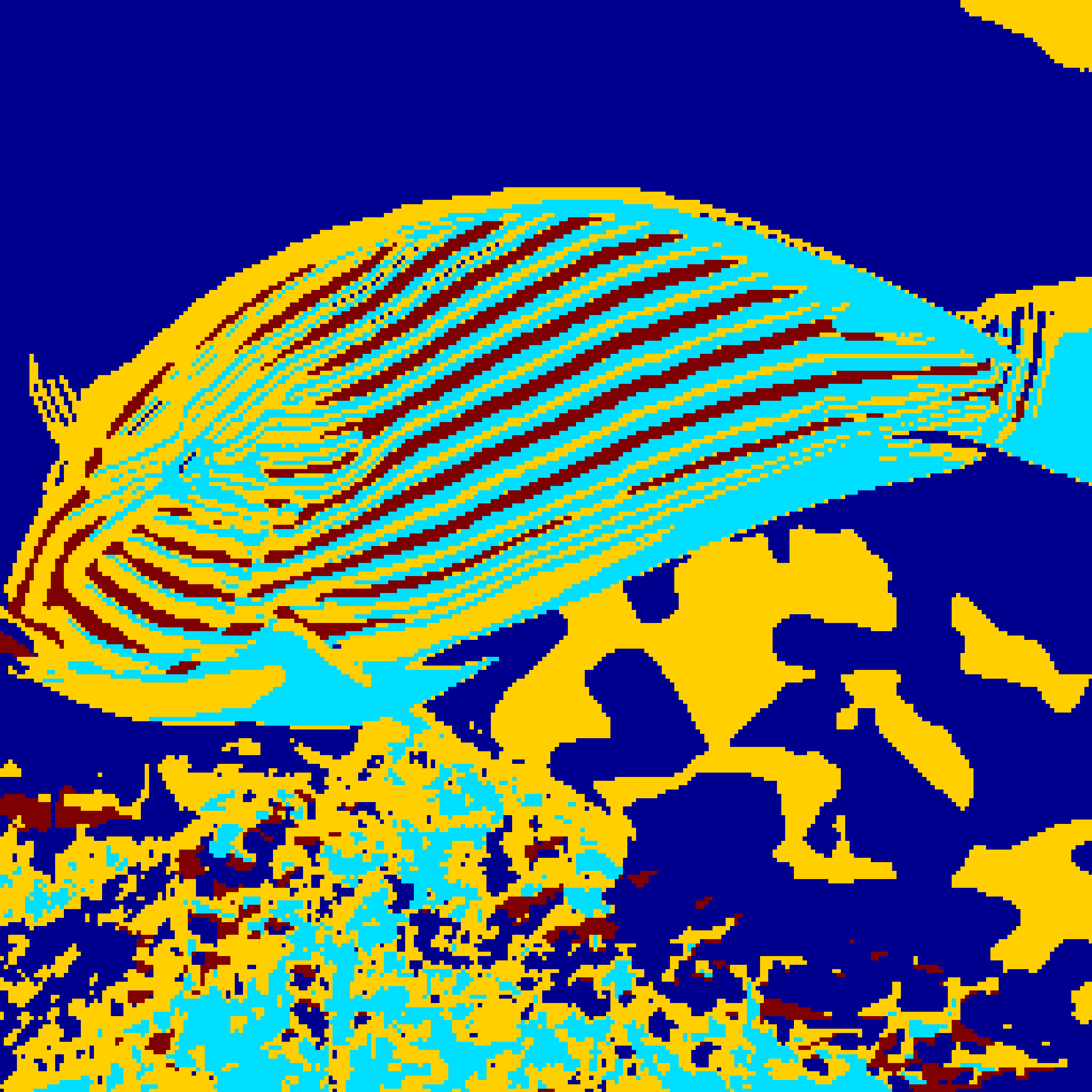}
	}
	\hspace{1cm}
	\subfigure[Shearlet segmentation, noisy image, \newline $\gamma = 4$, $\lambda = \frac{1}{100}(0,0.1,0.2,0.4,2)$,\newline 30 iterations]
	{\label{fig:fishNoisyShearlet}
	  \includegraphics[width=0.38\textwidth]{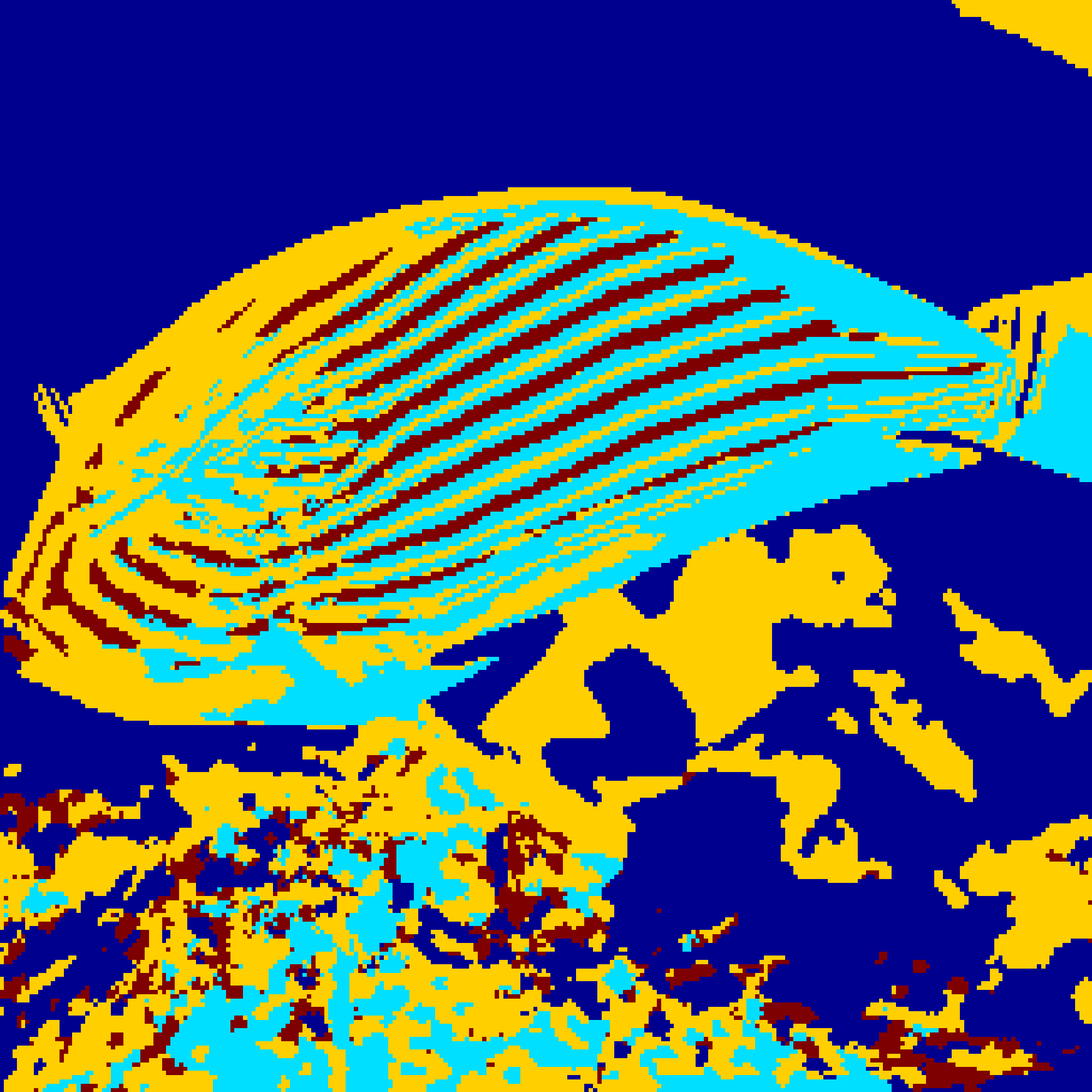}
	}
	\caption{Four-class segmentation of a clown-surgeon fish image}
	\label{fig:fishSegmentation}
\end{figure}
In summary we have seen that shearlets can be successfully applied for the segmentation of curved textures
in conjunction with a convex multi-label model and the ADMM algorithm.

\end{document}